\documentclass[a4paper]{amsart}
\usepackage{amsmath,amsthm,amssymb,latexsym,epic,bbm,comment, mathrsfs}
\usepackage{graphicx,enumerate,stmaryrd,color}
\usepackage[all,2cell]{xy}
\xyoption{2cell}

\usepackage[backend=bibtex,style=alphabetic,sorting=nty,doi=false,url=false,eprint=true]{biblatex}
\theoremstyle{definition}
\newtheorem{theorem}{Theorem}
\newtheorem{lemma}[theorem]{Lemma}
\newtheorem{corollary}[theorem]{Corollary}
\newtheorem{proposition}[theorem]{Proposition}

\newtheorem{example}[theorem]{Example}

\newtheorem{definition}[theorem]{Definition}
\usepackage[all]{xy}
\usepackage[active]{srcltx}
\usepackage[parfill]{parskip}
\usepackage{graphicx,enumerate}
\usepackage[normalem]{ulem}
\usepackage{tikz}
\newcommand{\tto}{\twoheadrightarrow}
\usepackage{fullpage}
\usepackage{faktor}
\usepackage{colortbl}
\newcommand{\blank}{\underline{\hspace{0.2cm}}}
\DeclareMathOperator*\medoplus{\mathchoice
	{\textstyle\bigoplus}
	{\textstyle\bigoplus}
	{\scriptstyle\bigoplus}
	{\scriptscriptstyle\bigoplus}
}
\usepackage{array}

\begin{document}
	\title[Tilting modules and exceptional sequences]
	{Tilting modules and exceptional sequences \\ for a family of dual extension algebras}
	
	\author{Elin Persson Westin and Markus Thuresson}
	
	\begin{abstract}
		We provide a classification of generalized tilting modules and full exceptional sequences for the dual extension algebra of the path algebra of a uniformly oriented linear quiver modulo the ideal generated by paths of length two with its opposite algebra. 
		For the classification of generalized tilting modules we develop a  combinatorial model for the poset of indecomposable self-orthogonal modules with standard filtration with respect to the relation arising from higher extensions. 
	\end{abstract}
	
	\maketitle
	
	
	
	\section{Introduction}
	Since its introduction in \cite{BB80, HR82}, tilting theory has become an important part of the representation theory of finite-dimensional algebras. A basic classification problem in tilting theory is to classify all generalized tilting modules over a given algebra. In general, this problem is very difficult. Some instances of where this problem and its generalizations have been studied can be found in \cite{Ad16, BK04, MU93, PW20, Ya12}.
	
	Quasi-hereditary algebras were first defined in \cite{Sc87}. In \cite{CPS88}, highest weight categories were introduced as a category theoretical counterpart of certain structures in the representation theory of complex semisimple Lie algebras. In the same paper, \cite{CPS88}, the quasi-hereditary algebras were characterized as exactly those finite-dimensional algebras whose module categories are highest weight categories. Examples of quasi-hereditary algebras include hereditary algebras, Schur algebras, algebras of global dimension two and algebras describing blocks of BGG category $\mathcal{O}$. 
	
	The chief protagonists of the representation theory of quasi-hereditary algebras are the standard and costandard modules, as well as the associated subcategories, $\mathcal{F}(\Delta)$ and $\mathcal{F}(\nabla)$, of the module category, consisting of those modules which admit a filtration by standard and costandard modules, respectively. Importantly, Ringel showed in \cite{Ri91} that $(\mathcal{F}(\Delta),\mathcal{F}(\nabla))$ is a homologically orthogonal pair.
	Moreover, Ringel showed that for any quasi-hereditary algebra, there exists a generalized tilting module $T$, called the characteristic tilting module, whose additive closure equals exactly $\mathcal{F}(\Delta)\cap \mathcal{F}(\nabla)$.
	
	Originally studied in \cite{Bo89}, and related to tilting theory, are exceptional modules and full exceptional sequences over a given algebra. For a quasi-hereditary algebra, the sequences of standard and costandard modules are examples of full exceptional sequences. The papers \cite{HP19, PW20} provide a classification of the full exceptional sequences over the Auslander algebra of $\Bbbk[x]/(x^n)$ and its quadratic dual, respectively. 
	
	A large family of quasi-hereditary algebras is constituted by the dual extension algebras defined by Xi in \cite{Xi}. 
	These were introduced as an example of a class of BGG algebras. BGG algebras are quasi-hereditary algebras admitting a simple preserving duality on their module categories, see \cite{Ir90}. Xi's original construction was soon generalized and has been studied in greater detail in \cite{DengXi, DengXiringel, LW15, Lixu, Shun, Xigldim, XiTilting}. The algebras studied in this paper are examples of dual extension algebras.
	
	The following is a brief description of the main results of this article.
	\begin{enumerate}[(A)]
		\item Let $\mathbb{A}_n$ be the uniformly oriented linear quiver with $n$ vertices, let $\Bbbk$ an algebraically closed field and set $$A_n=\faktor{\Bbbk \mathbb{A}_n}{ (\operatorname{rad}\Bbbk \mathbb{A}_n )^2}.$$ 
		Let $\Lambda_n$ be the dual extension algebra $\mathcal{A}\left(A_n,A_n^{\operatorname{op}}\right)$.
		We obtain a classification of generalized tilting modules over the algebra $\Lambda_n$. This is achieved through the following steps. First, we show that any generalized tilting module is contained in $\mathcal{F}(\Delta)$ or in $\mathcal{F}(\nabla)$. 
		Using the simple-preserving duality, this reduces the problem to the classification of generalized tilting modules in $\mathcal{F}(\Delta)$.
		Then, we consider the set of isomorphism classes of indecomposable self-orthogonal modules with a standard filtration. 
		We provide a combinatorial description of the non-zero extensions of positive degree between the modules in this set in terms of a certain transitive relation.  
		This interpretation allows us to classify generalized tilting modules contained in $\mathcal{F}(\Delta)$ as the maximal anti-chains with respect to this fixed relation.
		\item We obtain a classification of full exceptional sequences of $\Lambda_n$-modules. We show that a full exceptional sequence is uniquely determined by going through each index $i$, for $2\leq i \leq n$, and choosing either the standard module $\Delta(i)$ or the costandard module $\nabla(i)$. More precisely, a full exceptional sequence is of the form	
		$$\left(\nabla(m_1),\nabla(m_2),\dots, \nabla(m_i),L(1), \Delta(n_1),\Delta(n_2),\dots, \Delta(n_j)\right)$$
		where
		\begin{enumerate}[$\bullet$]
			\item $i+j=n-1$;
			\item $\{m_1,m_2,\dots, m_i, n_1,n_2,\dots, n_j\}=\{2,3,\dots, n\}$;
			\item $m_1 > m_2 >\dots >m_i$ and $n_1<n_2<\dots n_j$.
		\end{enumerate}
	\end{enumerate}
	Comparing the results with the first author's article, \cite{PW20}, in which the same problems are studied for another class of algebras, we see that the classification of generalized tilting modules is more complicated in the current case, while the classification of full exceptional sequences is identical in the two cases, in the sense that the form of the sequences is the same.

	The present article is organized as follows. In Section~2, we briefly introduce the algebras $\Lambda_n$, which are our objects of study, and recall some results on their quasi-hereditary structure. In Section~3, we classify the indecomposable $\Lambda_n$-modules using the results from \cite{BR87, WaldWasch} on the classification of indecomposable modules over special biserial algebras. We also introduce some notation which is important for the readability of the subsequent sections. Section~4 contains a classification of the self-orthogonal indecomposable $\Lambda_n$-modules. In Section~5, we classify the generalized tilting modules over $\Lambda_n$. Section~6 contains the classification of full exceptional sequences over $\Lambda_n$.

	\section{Background}
	Throughout the rest of the article, let $\Bbbk$ be an algebraically closed field. 
	Let $\mathbb{A}_n$ be the uniformly oriented linear quiver $$\xymatrix{1\ar[r] & 2 \ar[r] & \dots \ar[r] & n-1 \ar[r] & n},$$ for some $n\in\mathbb{Z}_{>1}$ and denote by $\Bbbk \mathbb{A}_n$ the corresponding path algebra. Let $A_n$ be the quotient of $\Bbbk\mathbb{A}_n$ by the ideal $\left( \operatorname{rad} \Bbbk\mathbb{A}_n\right)^2$. Finally, we define $\Lambda_n$ to be the dual extension algebra of $A_n$ with its opposite algebra, $A_n^{\operatorname{op}}$, that is, $\Lambda_n=\mathcal{A}(A_n, A_n^{\operatorname{op}})$. Then, $\Lambda_n$ is given by the quiver
	
	\begin{align*}
		\xymatrix{
			1 \ar@/^1pc/[r]^-{\alpha_1} & 2	\ar@/^1pc/[l]^-{\alpha^\prime_1} \ar@/^1pc/[r]^-{\alpha_2} & 3 \ar@/^1pc/[l]^-{\alpha^\prime_2} \ar@/^1pc/[r] & \dots \ar@/^1pc/[l] \ar@/^1pc/[r] & n-1 \ar@/^1pc/[r]^-{\alpha_{n-1}} \ar@/^1pc/[l] & n \ar@/^1pc/[l]^-{\alpha_{n-1}^\prime}
		}
	\end{align*}
	subject to the relations
	\begin{align*}
		\alpha_{i+1}\alpha_i=0,\quad \alpha_j^\prime \alpha_{j+1}^\prime=0\quad \text{and}\quad \alpha_i \alpha_i^\prime=0.
	\end{align*}

	Let $\Lambda_{n}\operatorname{-mod}$ denote the category of finite-dimensional left $\Lambda_{n}$-modules. Throughout the rest of the article, we take ``module'' to mean left module.
	The algebra $\Lambda_n$ has a simple-preserving duality on its module category, denoted by $\star$, induced by the antiautomorphism given by swapping the arrows $\alpha_i$ and $\alpha_i'$ in the quiver of  $\Lambda_n$.
	
	Let $L(i)$, where $\ 1\leq i\leq n$, denote the simple $\Lambda_n$-module corresponding to the vertex $i$. Let $P(i)$ and $I(i)$ denote its projective cover and injective envelope, respectively.
	
	\begin{definition}\cite{CPS88}
		Let $\Lambda$ be a finite-dimensional algebra. Let $\{1,\dots, n\}$ be an indexing set for the isomorphism classes of simple $\Lambda$-modules and let $<$ be a partial order on $\{1,\dots, n\}$. The algebra $\Lambda$ is said to be \emph{quasi-hereditary} with respect to $<$ if there exist modules $\Delta(i)$, where $i\in \{1,\dots, n\}$, called \emph{standard modules}, satisfying the following.
		\begin{enumerate}[(QH1)]
			\item  There is a surjection $P(i)\tto \Delta(i)$ whose kernel admits a filtration with subquotients $\Delta(j)$, where $j>i$.
			\item There is a surjection $\Delta(i) \tto L(i)$ whose kernel admits a filtration with subquotients $L(j)$, where $j<i$.
		\end{enumerate}
		This is equivalent to the existence of modules $\nabla(i)$, where $i\in \{1, \dots, n\}$, called \emph{costandard modules}, satisfying the following.
		\begin{enumerate}[(QH1)$^\prime$]
			\item There is an injection $\nabla(i)\hookrightarrow I(i)$ whose cokernel admits a filtration with subquotients $\nabla(j)$, where $j>i$.
			\item There is an injection $L(i)\hookrightarrow \nabla(i)$ whose cokernel admits a filtration with subquotients $L(j)$, where $j<i$.
		\end{enumerate}
	\end{definition}

	It is easy to see that $\Lambda_n$ is quasi-hereditary with respect to the natural ordering on $\{1,\dots, n\}$. Indeed, (QH1) and (QH2) are easily verified with standard and costandard modules as below. Note that results from \cite{Xi} show that the dual extension algebra $\Lambda_n=\mathcal{A}(A_n,A_n^{\operatorname{op}})$ is quasi-hereditary. Throughout, let $\mathcal{F}(\Delta)$ denote the full subcategory of $\Lambda_n\operatorname{-mod}$ consisting of those modules which admit a filtration by standard modules. Similarly, let $\mathcal{F}(\nabla)$ denote the full subcategory consisting of those modules which admit a filtration by costandard modules. Since $\Lambda_n$ is quasi-hereditary, Theorem~5 from \cite{Ri91} guarantees that there exists a basic module $T$, called the \emph{characteristic tilting module}, such that the additive closure of $T$, denoted by $\operatorname{add}T$, equals $\mathcal{F}(\Delta)\cap \mathcal{F}(\nabla)$. Moreover, $T$ has the same number of non-isomorphic indecomposable summands as the number of isomorphism classes of simple modules, and we write $$T=\medoplus\limits_{k=1}^n T(k).$$ The indecomposable direct summand $T(k)$ of $T$ is uniquely determined (up to isomorphism) by the property that it belongs to $\mathcal{F}(\Delta)\cap \mathcal{F}(\nabla)$ and that there exists a monomorphism $\Delta(k)\hookrightarrow T(k)$, whose cokernel admits a filtration by standard modules.
	
	We conclude this section by drawing the Loewy diagrams of the structural modules of $\Lambda_n$ and stating some of their elementary properties.

	\begin{align*}
		P(i): \; &\vcenter{ \xymatrixrowsep{0.2cm}\xymatrixcolsep{0.2cm}\xymatrix{ & \ar[ld] i\ar[rd] \\ i-1 & & i+1 \ar[ld] \\ & i}} \;
		& 
		I(i): \; &\vcenter{ \xymatrixrowsep{0.2cm}\xymatrixcolsep{0.2cm}\xymatrix{ & i\ar[rd] \\ i-1 \ar[rd]& & i+1 \ar[ld] \\ & i}} \;
		& 
		\text{for } &i=2,\dots, n-1, \\
		\Delta(i): \; & \vcenter{\xymatrixrowsep{0.2cm}\xymatrixcolsep{0.2cm}\xymatrix{ & i \ar[ld] \\ i-1}} \;
		& 
		\nabla(i): \; & \vcenter{\xymatrixrowsep{0.2cm}\xymatrixcolsep{0.2cm}\xymatrix{ i-1 \ar[rd] \\  &i }}  \;
		& 
		\text{for } &i=2,\dots, n.
	\end{align*}
	The remaining cases are $\Delta(1)=\nabla(1)=L(1)$, $P(n)=\Delta(n)$, $\ I(n)=\nabla(n)$ and
	$$P(1)=I(1): \vcenter{ \xymatrixrowsep{0.2cm}\xymatrixcolsep{0.2cm}\xymatrix{ 1 \ar[rd] \\ &2 \ar[ld] \\ 1}}.$$

	Finally, from these pictures we see that we have the following lemma.
	\begin{lemma}\label{lemma: short exact sequences}
		For every $i=1,\dots, n-1$ we have the following non-split short exact sequences in $\Lambda_n\operatorname{-mod}$:
		\[
		\Delta(i+1) \hookrightarrow P(i) \tto \Delta(i), \quad \nabla(i)\hookrightarrow I(i) \tto \nabla(i+1).
		\]
	\end{lemma}
	\begin{proposition}\label{proposition:ext between standards and costandards}
		For $m=j-i>0$ we have 
		\[
		\operatorname{Ext}_{\Lambda_n}^m(\Delta(i),\Delta(j))\neq 0 \quad \text{and}\quad \operatorname{Ext}_{\Lambda_n}^m(\nabla(j),\nabla(i))\neq 0.
		\]
	\end{proposition}
	\begin{proof}
		Using the simple-preserving duality, the first inequality implies the second. By applying the functor $\operatorname{Hom}_{\Lambda_n}(\Delta(i),\blank)$ to the short exact sequence
		$$\Delta(j) \hookrightarrow P(j) \tto \Delta(j-1),$$
		and considering the resulting long exact sequence, we get $\operatorname{Ext}_{\Lambda_n}^{j-i}(\Delta(i),\Delta(j))\cong \operatorname{Ext}_{\Lambda_n}^{j-i-1}(\Delta(i),\Delta(j-1))$. Repeating this argument, we get $$\operatorname{Ext}_{\Lambda_n}^{j-i}(\Delta(i),\Delta(j))\cong \operatorname{Ext}_{\Lambda_n}^{j-i-1}(\Delta(i),\Delta(j-1))\cong \dots \cong \operatorname{Ext}_{\Lambda_n}^1(\Delta(i),\Delta(i+1))\neq 0,$$
		where, in the last step, we use that the short exact sequences in Lemma \ref{lemma: short exact sequences} are non-split.
	\end{proof}
	
	\section{Indecomposable $\Lambda_n$-modules}
	In order to classify indecomposable $\Lambda_n$-modules, we use the fact that the algebra $\Lambda_n$ is a string algebra. For these algebras, the classification is known. 
	\begin{definition}\cite{WaldWasch}
		Let $\Lambda=\faktor{\Bbbk Q}{I}$ be the quotient of the path algebra of the quiver $Q=(Q_0,Q_1)$ by some admissible ideal $I$. For an arrow $\alpha\in Q_1$, denote by $s(\alpha)$ and $t(\alpha)$ the source and target vertex of $\alpha$, respectively.  Then $\Lambda$ is called \emph{special biserial} if the following hold.
		
		\begin{enumerate}[(SB1)]
			\item For each vertex $i$, there are at most two arrows with $i$ as its source, and at most two arrows with $i$ as its target.
			\item For $\alpha, \beta,\gamma \in Q_1$ such that $t(\alpha)=t(\beta)=s(\gamma)$ and $\alpha\neq \beta$, we have $\gamma \alpha \in I$ or $\gamma\beta \in I$.
			\item For $\alpha,\beta,\gamma \in Q_1$ such that $s(\alpha)=s(\beta)=t(\gamma)$ and $\alpha\neq\beta$ we have $\alpha\gamma\in I$ or $\beta\gamma \in I.$
		\end{enumerate}
		If, in addition, the ideal $I$ is generated by zero relations, $\Lambda$ is called a \emph{string algebra}.
	\end{definition}
	We immediately note that $\Lambda_n$ is a string algebra for all $n\in \mathbb{Z}_{>1}$. For special biserial algebras and string algebras the classification of indecomposable modules is known, see \cite{BR87, WaldWasch}. There exist two classes of indecomposable modules, the so-called string modules and band modules. We will show that in the case of $\Lambda_n$, there are no band modules and therefore a complete list of the indecomposable $\Lambda_n$-modules is given by the string modules.
	
	We follow closely the notation of \cite{WaldWasch}. Let $L=(L_0,L_1)$ denote the quiver
	
	$$L=\xymatrix{
		1 \ar@{-}[r]^-{a_1} & 2 \ar@{-}[r]^-{a_2}& \dots \ar@{-}[r]^-{a_{r-1}} & r \ar@{-}[r]^-{a_r} & r+1
	},\quad r\geq 0$$
	where the $a_i$ are arrows with either orientation. Define a map $\varepsilon:L_1\to \{-1,1\}$ by

	$$\varepsilon(a_i)=\left\{\begin{array}{r l}
		1, & \text{if } a_i:i\to i+1; \\
		-1, & \text{if } a_i :i+1\to i.
	\end{array}\right.$$
	Similarly, we denote by $Z=(Z_0,Z_1)$ the quiver
	$$Z=\xymatrix{
		\overline{1} \ar@/_2pc/@{-}[rrr]_-{b_{\overline{r}}} \ar@{-}[r]^-{b_{\overline{1}}} & \overline{2} \ar@{-}[r]^-{b_{\overline{2}}}& \dots \ar@{-}[r]^-{b_{{\overline{r-1}}}} & \overline{r}
	},\quad r\geq 2$$
	where the $b_{\overline{i}}$ are arrows with either orientation and where $\overline{i}$ denotes the congruence class of $i$ modulo $r$. Again, define $\varepsilon: Z_1\to \{-1,1\}$ by 
	$$\varepsilon(b_{\overline{i}})=\left\{\begin{array}{r l}
		1, & \text{if } b_{\overline{i}}: \overline{i}\to \overline{i+1};\\
		-1, & \text{if } b_{\overline{i}}: \overline{i+1}\to \overline{i}.
	\end{array}\right.$$
	Let $Q$ be some quiver. A quiver homomorphism $w:L\to Q$  is called a \emph{walk} of length $r$ in $Q$. A walk is called a \emph{path} if $\varepsilon(a_i)=1$ for all $i$. Similarly, a homomorphism $u:Z\to Q$ is called a \emph{tour} in $Q$. A tour is called a \emph{circuit} if $\varepsilon(b_{\overline{i}})=1$ for all $i$. The restriction of $v$ (or $u$) to a connected linear subquiver $L^\prime$ of $L$ (or $Z$) is called a \emph{subwalk} or a \emph{subpath} (or, a \emph{subtour} or \emph{subcircuit}).
	
	\begin{definition}\cite{WaldWasch}
		Fix a quiver $Q$ and an admissible ideal $I\subset \Bbbk Q$. A walk $v:L\to Q$ is called a \emph{$V$-sequence} if the following hold.
		
		\begin{enumerate}[(VS1)]
			\item Each subpath of $v$ does not belong to $I$.
			\item If $\varepsilon(a_i)\neq \varepsilon(a_{i+1})$, then $v(a_i)\neq v(a_{i+1})$.
		\end{enumerate}
		Similarly, a tour $u:Z\to Q$ is called a \emph{primitive $V$-sequence} if the following hold.
		\begin{enumerate}[(VS1)]
			\setcounter{enumi}{2}
			\item The tour $u$ is not a circuit and each subpath of $u$ does not belong to $I$.
			\item If $\varepsilon(b_i)\neq\varepsilon(b_{\overline{i+1}})$, then $u(b_{\overline{i}})\neq u(b_{\overline{i+1}})$.
			\item There is no automorphism $\sigma\neq \operatorname{id}$ of $Z$, permuting the vertices cyclically such that $u=u\circ \sigma$.
		\end{enumerate}
	\end{definition}
	In \cite{WaldWasch}, the authors show that we can obtain all indecomposable modules over a special biserial algebra from $V$-sequences and primitive $V$-sequences. These correspond exactly to the string and band modules, respectively. In Proposition \ref{proposition:V-sequences - no primitive V-sequences}, we will see that there are no primitive $V$-sequences $u:Z\to Q$, and consequently, no band modules. To obtain an indecomposable module from a $V$-sequence $v:L\to Q$, consider the following representation of the bound quiver $(Q,I)$. At each vertex $x\in Q_0$, we put the vector space $\Bbbk$ if $x$ is in the image of $v$, and the zero space otherwise. At each arrow $\xymatrixcolsep{0.3cm}\xymatrix{ x \ar[r]^-{\alpha} & y}\in Q_1$, we put the identity map on the vector space $\Bbbk$ if $\alpha$ is in the image of $v$, and the zero map otherwise. This representation is then equivalent to a $\Bbbk Q/I$-module.
	
	However, an indecomposable module $M$ does not arise from a unique $V$-sequence. In fact, the indecomposable module $M$ corresponding to a $V$-sequence $v: L\to Q$ is isomorphic to the indecomposable module $M^\prime$ corresponding to a $V$-sequence $v^\prime: L^\prime \to Q$ if and only if there is a quiver isomorphism $\sigma: L^\prime \to L$ such that $v^\prime = v \circ \sigma$. In this situation, we say that the $V$-sequences $v$ and $v^\prime$ are isomorphic. There are only two possibilities for such an isomorphism $\sigma$. The first possibility is that $L=L^\prime$ and $\sigma=\operatorname{id}$. The second possibility is that $\sigma$ acts on the vertices of $L^\prime$ by $\sigma(i)=r+2-i$, that is, $\sigma$ swaps the vertex $1$ and the vertex $r+1$, the vertex $2$ and the vertex $r$, and so on. Here, we must have an arrow $\xymatrixcolsep{0.5cm}\xymatrix{i\ar[r]^-{\alpha^\prime} & j}$ in $L^\prime$ if and only if we have an arrow $\xymatrixcolsep{0.5cm}\xymatrix{ r+2-i \ar[r]^-{\alpha} & r+2-j}$ in $L$. The corresponding $V$-sequence $v^\prime:L^\prime\to Q$ must then be given by $v^\prime(i)=v(r+2-i)$. 
	
	\begin{example}
		Let $L$ be the quiver $\xymatrixcolsep{0.5cm}\xymatrix{1 \ar[r]^-\alpha & 2 & 3 \ar[r]^-\gamma \ar[l]_-\beta & 4}$. Then $L^\prime$ as described above is the quiver $\xymatrixcolsep{0.5cm}\xymatrix{1 & 2 \ar[l]_-{\gamma^\prime} \ar[r]^-{\beta^\prime} & 3 & \ar[l]_-{\alpha^\prime}4}.$ In a picture, the isomorphism $\sigma:L^\prime\to L$ is as follows.
		$$
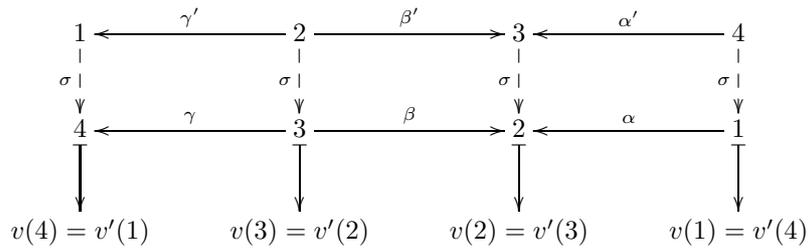
\xymatrix{1 \ar@{-->}[d]_-{\sigma} & 2 \ar[l]_-{\gamma^\prime} \ar@{-->}[d]_-{\sigma} \ar[r]^-{\beta^\prime} & 3 \ar@{-->}[d]_-{\sigma} & \ar@{-->}[d]_-{\sigma}\ar[l]_-{\alpha^\prime}4 \\
			4 \ar@{|->}[d] & \ar[l]_-{\gamma} \ar[r]^-\beta3 \ar@{|->}[d]& \ar@{|->}[d]2 & \ar[l]_-{\alpha}1 \ar@{|->}[d] \\
			v(4)=v^\prime(1) & v(3)=v^\prime(2) & v(2)=v^\prime(3) & v(1)=v^\prime(4)}$$
	\end{example}
	
	\begin{proposition}\label{proposition:V-sequences - no primitive V-sequences}
		Let $Q$ be the quiver of $\Lambda_n$ and let $I\subset \Bbbk Q$ be the ideal generated by the relations in Section 1. Then, there are no primitive $V$-sequences $u:Z\to Q$. Moreover, the only $V$-sequences $v:L\to Q$, up to isomorphism, are of one of the following forms.
		
		\begin{enumerate}[(a)]
			\item We have $\varepsilon(a_i)\neq \varepsilon(a_{i+1})$ and $v(i+1)=v(i)+1$ for all $i$. The $V$-sequence is given by the following picture.
			$$
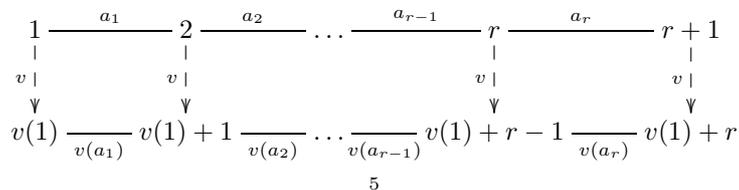
\xymatrix{
				1 \ar@{-->}[d]_-{v} \ar@{-}[r]^-{a_1} & 2 \ar@{-->}[d]_-{v}\ar@{-}[r]^-{a_2}& \dots \ar@{-}[r]^-{a_{r-1}} & r\ar@{-->}[d]_-{v} \ar@{-}[r]^-{a_r} & r+1 \ar@{-->}[d]_-{v} \\
				v(1) \ar@{-}[r]_-{v(a_1)} & v(1)+1 \ar@{-}[r]_-{v(a_2)}& \dots  \ar@{-}[r]_-{v(a_{r-1})} & v(1)+r-1 \ar@{-}[r]_-{v(a_r)} & v(1)+r
			}$$
			\item We have $\varepsilon(a_{s})=\varepsilon(a_{s+1})=1$ for some $s$, $\varepsilon(a_i)\neq \varepsilon(a_{i+1})$ for all $i\neq s$, $v(i+1)=v(i)+1$ for $i\leq s$ and $v(i+1)=v(i)-1$ for $i> s$. The $V$-sequence is given by the following picture.
			$$\xymatrix{
				1 \ar@{-->}[d]_-{v} \ar@{-}[r]^-{a_1} &  \dots \ar@{-}[r]^-{a_{s-1}} & s \ar@{-->}[d]_-{v}\ar@{-}[r]^-{a_{s}} &\ar@{-->}[d]_-{v} s+1 \ar@{-}[r]^-{a_{s+1}}  & \ar@{-->}[d]_-{v}s+2 \ar@{-}[r]^-{a_{s+2}} & \dots \ar@{-}[r]^-{a_{r}}& \ar@{-->}[d]_-{v}r+1 \\
				v(1) \ar@{-}[r]^-{a_1} &  \dots \ar@{-}[r]^-{a_{s-1}} & v(1)+s-1 \ar@{-}[r]^-{a_{s}} & v(1)+s \ar@{-}[r]^-{a_{s+1}}  & v(1)+s-1 \ar@{-}[r]^-{a_{s+2}} & \dots \ar@{-}[r]^-{a_{r}}& v(1)+2s-r
			}$$
		\end{enumerate}
	\end{proposition}
	\begin{proof}
		Let $v: L\to Q$ be a $V$-sequence, where $v(i)=j$, and consider the subquiver
		$$\xymatrix{
			i\ar@{-}[r]^-{a_i}	& i+1 \ar@{-}[r]^-{a_{i+1}} & i+2
		}$$
		of $L$. There are four cases, depending on the values of $\varepsilon(a_i)$ and $\varepsilon(a_{i+1})$.
		\begin{enumerate}[(I)]
			\item Assume that $\varepsilon(a_i)=\varepsilon(a_{i+1})=1.$ Then, our subquiver looks like $\xymatrix{i\ar[r]^-{a_i} & i+1\ar[r]^-{a_{i+1}} & i+2}$. By assumption, $v(a_i)$ is an arrow starting in $j$, so that $v(a_i)=\alpha_j$ or $v(a_i)=\alpha_{j-1}^\prime$. If $v(a_i)=\alpha_{j-1}^\prime$, then, we immediately get that 
			$$v(a_{i+1})=\alpha_{j-1}\quad \text{or}\quad v(a_{i+1})=\alpha_{j-2}^\prime,$$
			and in both cases, $v(a_{i+1})v(a_i)=0$, contradicting (VS1).
			
			If $v(a_i)=\alpha_j$, then, we have
			$$v(a_{i+1})=\alpha_{j+1}\quad\text{or}\quad v(a_{i+1})=\alpha_j^\prime.$$
			If $v(a_{i+1})=\alpha_{j+1}$, we again contradict (VS1). Therefore, we conclude that the subquiver $$\xymatrix{i\ar[r]^-{a_i} & i+1\ar[r]^-{a_{i+1}} & i+2}$$
			is mapped to the following subquiver.
			$$\xymatrixrowsep{0.5cm}\xymatrixcolsep{0.2cm}\xymatrix{
				v(i)=j \ar[rd]^-{\alpha_j} \\
				& v(i+1)=j+1 \ar[ld]^-{\alpha_j^\prime }\\
				v(i+2)=j
			}$$
			\item Assume that $\varepsilon(a_i)=\varepsilon(a_{i+1})=-1$. Then, our subquiver looks like $\xymatrix{ i & i+1 \ar[l]_-{a_i} & i+2 \ar[l]_-{a_{i+1}}}$. A similar argument as in the previous case shows that this is mapped to the following subquiver.
			$$\xymatrixrowsep{0.5cm}\xymatrixcolsep{0.2cm}\xymatrix{
				v(i+2)=j \ar[rd]^-{\alpha_j} \\
				& v(i+1)=j+1 \ar[ld]^-{\alpha_j^\prime }\\
				v(i)=j
			}$$
			\item Assume that $\varepsilon(a_i)=1$ and $\varepsilon(a_{i+1})=-1$. Then, our subquiver looks like $\xymatrix{i\ar[r]^-{a_i} & i+1 & i+2 \ar[l]_-{a_{i+1}}}$. Again, we have $v(a_i)=\alpha_j$ or $v(a_i)=\alpha_{j-1}^\prime$. If $v(a_i)=\alpha_j$,  we have $v(a_{i+1})=\alpha_{j+1}^\prime$ since, by (VS2), $v(a_i)\neq v(a_{i+1})$, and we get the following picture.
			$$\xymatrixrowsep{0.2cm}\xymatrixcolsep{0.5cm}\xymatrix{
				v(i)=j \ar[rd]^-{\alpha_j} & & v(i+2)=j+2 \ar[ld]_-{\alpha_{j+1}^\prime} \\
				& v(i+1)=j+1	
			}$$
			If, instead, $v(a_i)=\alpha_{j-1}^\prime$, we get $v(a_{i+1})=\alpha_{j-2}$ and the following picture.
			$$\xymatrixrowsep{0.2cm}\xymatrixcolsep{0.5cm}\xymatrix{
				v(i+2)=j-2 \ar[rd]^-{\alpha_{j-2}} & & v(i)=j \ar[ld]_-{\alpha_{j-1}^\prime} \\
				& v(i+1)=j-1	
			}$$
			\item Assume that $\varepsilon(a_i)=-1$ and $\varepsilon(a_{i+1})=1$. Then, our subquiver looks like $\xymatrix{i & i+1 \ar[l]_-{a_i} \ar[r]^-{a_{i+1}}& i+2}$. By similar arguments as in the previous case, we get one of the following two possible pictures.
			
			$$\xymatrixrowsep{0.5cm}\xymatrixcolsep{0.5cm}
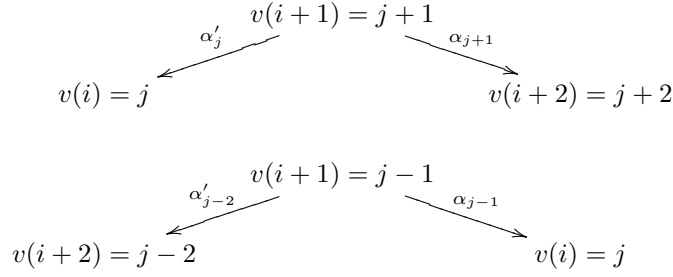
\xymatrix{
				& v(i+1)=j+1 \ar[rd]^-{\alpha_{j+1}} \ar[ld]_-{\alpha_j^\prime} \\
				v(i)=j & & v(i+2)=j+2	\\
				& v(i+1)=j-1 \ar[rd]^-{\alpha_{j-1}} \ar[ld]_-{\alpha_{j-2}^\prime} \\
				v(i+2)=j-2 & & v(i)=j
			}$$
		\end{enumerate}
		
		Let $u: Z\to Q$ be a primitive $V$-sequence and let $\overline{a}$ be such that $u(\overline{a})=j\leq u(\overline{s})$ for all $s\in Z_0$. It follows that $u(\overline{a+1})=j+1=u(\overline{a-1})$. Then, to not contradict (VS4), the subquiver
		$$\xymatrix{ \overline{a-1} \ar@{-}[r] & \overline{a} \ar@{-}[r] & \overline{a+1}}$$
		of $L$ is mapped to the subquiver
		$$\xymatrixrowsep{0.5cm}\xymatrixcolsep{0.2cm}
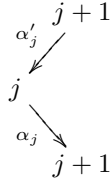
\xymatrix{
			& j+1 \ar[ld]_-{\alpha_j^\prime}\\
			j \ar[rd]_-{\alpha_j} \\
			& j+1
		}$$
		of $Q$. But this configuration contradicts (VS3). We conclude that there are no primitive $V$-sequences.
		
		Let $v:L\to Q$ be a $V$-sequence. From (I)-(IV), we know that if $v(i+1)=v(i)-1$, then $v(i+2)=v(i)-2$, $v(i+3)=v(i)-3$ and so on. In this case $\varepsilon(a_s)\neq\varepsilon(a_{s+1})$ for all $s\geq i$. In particular, $\varepsilon(a_s)=\varepsilon(a_{s+1})$ can occur at most once in a $V$-sequence. There are two cases.
		
		\begin{enumerate}[(a)]
			\item We have $\varepsilon(a_i)\neq \varepsilon(a_{i+1})$, for all $i$. Then, either $v(i+1)=v(i)+1$ or $v(i+1)=v(i)-1$, for all $i$. However, any $V$-sequence of the latter type is isomorphic to one of the former type. This situation corresponds to part (a) of the statement of the proposition.

			\item We have $\varepsilon(a_{s})=\varepsilon(a_{s+1})$, for some $s$, $\varepsilon(a_i)\neq \varepsilon(a_{i+1})$, for all $i\neq s$, $v(i+1)=v(i)+1$, for $i\leq s$ and $v(i+1)=v(i)-1$, for $i> s$. Then, either $\varepsilon(a_s)=\varepsilon(a_{s+1})=1$, or $\varepsilon(a_s)=\varepsilon(a_{s+1})=-1$.  However, any $V$-sequence of the latter type is isomorphic to one of the former type. This situation corresponds to part (b) of the statement of the proposition. \qedhere
		\end{enumerate}
	\end{proof}

	\begin{definition}
		Define $\Omega(i,j,k)$, where $i,j\leq k$, to be the (up to isomorphism unique) indecomposable $\Lambda_n$-module with the following Loewy diagram.
		\begin{enumerate}[(a)]
			\item If $k\equiv i \mod 2$ and $k\equiv j\mod 2$:
			$$ \xymatrixrowsep{0.1cm} \xymatrixcolsep{0.5cm} \xymatrix{
				& i+1 \ar[dl] \ar[dr]& & & k-1 \ar[dl] \ar[ddr] \\
				i & & \ar@{}[r]|{\textstyle\dots} & & & \\
				& & & & & k \ar[ddl]\\
				j \ar[rd]& & \ar[dl] \ar@{}[r]|{\textstyle\dots} & \ar[rd] \\
				& j+1 & & & k-1
			}$$
			\item If $k\equiv i \mod 2$ and $k\not\equiv j \mod 2$:
			$$ \xymatrixrowsep{0.1cm} \xymatrixcolsep{0.5cm} \xymatrix{
				& i+1 \ar[dl] \ar[dr]& & & k-1 \ar[dl] \ar[ddr] \\
				i & & \ar@{}[r]|{\textstyle\dots} &  & & \\
				& & & & & k \ar[ddl]\\
				&  \ar[ld]j+1 \ar[rd]& & \ar[dl] \dots \ar[dr]\\
				j& & j+2 &  & k-1
			}$$
			\item If $k\not \equiv i \mod 2$ and $k\equiv j\mod 2$:
			$$ \xymatrixrowsep{0.1cm} \xymatrixcolsep{0.5cm} \xymatrix{
				i \ar[dr]& & \ar[dl] i+2 \ar[dr] && k-1 \ar[dl] \ar[ddr] \\
				& i+1& & \dots & & \\
				& & & & & k \ar[ddl]\\
				j \ar[rd]& & \ar[dl] \ar@{}[r]|{\textstyle\dots} & \ar[dr] \\
				& j+1 & & & k-1
			}$$
			\item If $k\not \equiv i \mod 2$ and $k\not\equiv j\mod 2$:
			$$ \xymatrixrowsep{0.1cm} \xymatrixcolsep{0.5cm} \xymatrix{
				i \ar[dr]& & \ar[dl] i+2 \ar[dr] & & k-1 \ar[dl] \ar[ddr] \\
				& i+1 & & \dots & & \\
				& & & & & k \ar[ddl]\\
				&  \ar[ld]j+1 \ar[rd] &&\ar[dl]  \dots \ar[dr]\\
				j& & j+2 &  & k-1
			}$$
		\end{enumerate}
		Note that, in case (a) and (b), we may have $i=k$, and, in case (a) and (c), we may have $j=k$. For all $i,j,k$, the simple preserving duality maps the module $\Omega(i,j,k)$ to the module $\Omega(j,i,k)$.
	\end{definition}

	\begin{proposition}
		The set $\{\Omega(i,j,k) \mid 1\leq i,j\leq k\leq n\}$
		is a complete and irredundant list of isomorphism classes of indecomposable $\Lambda_n$-modules. There are, in total, $\frac{n(n+1)(2n+1)}{6}$ isomorphism classes of indecomposable $\Lambda_n$-modules.
	\end{proposition}
	\begin{proof}
		By \cite{WaldWasch}, all indecomposable $\Lambda_n$-modules arise from $V$-sequences or primitive $V$-sequences. Using Proposition \ref{proposition:V-sequences - no primitive V-sequences} we know that there are no primitive $V$-sequences, and what the possible $V$-sequences look like. Let $v:L\rightarrow Q$ be a $V$-sequence, of the form (a) in Proposition \ref{proposition:V-sequences - no primitive V-sequences}, such that $v(1)=i$. If $\varepsilon(a_r)=1$, then the indecomposable module corresponding to $v$ is $\Omega(i,i+r, i+r)$. If instead $\varepsilon(a_r)=-1$, then the indecomposable module corresponding to $v$ is $\Omega(i+r,i, i+r)$.
		
		Let $v:L\rightarrow Q$ be a $V$-sequence, of the form (b) in Proposition \ref{proposition:V-sequences - no primitive V-sequences}, such that $v(1)=i$. Then the indecomposable module corresponding to $v$ is $\Omega(i,i+2s-r, i+s)$. Thus, we see that any $V$-sequence gives rise to a (unique) module of the form $\Omega(i,j,k)$, and that any such module may be obtained from a $V$-sequence. This proves the first part of the proposition.
		
		Every choice of $i$, $j$ and $k$ such that $1\leq i,j\leq k\leq n$ yields a unique module $\Omega(i,j,k)$. For a fixed $k$, there are $k$ choices of $i$, and $k$ choices of $j$, which implies that there are $k^2$ non-isomorphic modules $\Omega(i,j,k)$, with $k$ fixed. The total number of non-isomorphic indecomposable $\Lambda_n$-modules is therefore
		\begin{align*}
			\sum_{k=1}^n k^2 = \frac{n(n+1)(2n+1)}{6}.
		\end{align*} 
	\end{proof}

	For any subset $X\subseteq \{1,2,\dots, n\}$, we denote by $P_X$ the direct sum 
	$$P_X=\medoplus_{i\in X}P(i).$$
	Should $X$ be the empty set, we define $P_\emptyset:=0$.
	We use similar notation for such direct sums of other structural modules. For $a,b\in\{1,\dots, n\}$, with $a\leq b$, we fix the following notation.
	\begin{enumerate}[$\bullet$]
		\item When $a\equiv b\mod 2$, we put $[a,b]=\{a, a+2, \dots, b-2, b\}.$
		\item We put $[a, b)=\{c \in \{1,\dots, n\} \mid  a\leq c\leq b \text{ and }c\equiv a\mod 2\}$.
		\item We put $(a, b]=\{c \in \{1,\dots, n\} \mid a\leq c\leq b \text{ and }c\equiv b\mod 2\}$.
	\end{enumerate}
	For example,
	$$P_{[3,8)}=P(3)\oplus P(5)\oplus P(7)\quad \text{and}\quad P_{(3,8]}=P(4)\oplus P(6)\oplus P(8).$$
	Note that, if $a\equiv b\mod 2$, then $[a,b]=[a,b)=(a,b]$.

	\begin{definition}
		Define the \emph{upper arm} of the indecomposable $\Lambda_n$-module $M=\Omega(i,j,k)$, denoted by $\operatorname{upp}(M)$, as the quotient
		
		$$\operatorname{upp}(M)=\faktor{\Omega(i,j,k)}{\Omega(k-1, j, k-1)}\cong \Omega(i,k,k).$$
		
		Similarly, define the \emph{lower arm} of $M=\Omega(i,j,k)$, denoted by $\operatorname{low}(M)$, as the submodule
		
		$$\operatorname{low}(M)=\Omega(k,j,k) \subset \Omega(i,j,k).$$
	\end{definition}
	\begin{example}
		We draw the Loewy diagram of $M=\Omega(2,3,6):$ 
		$$ \xymatrixrowsep{0.1cm} \xymatrixcolsep{0.5cm} 
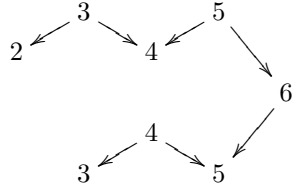
\xymatrix{
			& 3 \ar[dl] \ar[dr]& & 5 \ar[dl] \ar[ddr] \\
			2 & & 4 & & \\
			& & & & 6 \ar[ddl]\\
			& & 4 \ar[rd] \ar[dl] \\
			& 3 & & 5
		}$$
		Then, the Loewy diagrams of $\operatorname{upp}(2,3,6)$ and $\operatorname{low}(2,3,6)$ are
		$$\operatorname{upp}(M): \vcenter{\xymatrixrowsep{0.1cm} \xymatrixcolsep{0.5cm} 
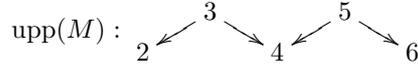
\xymatrix{
				& 3 \ar[dl] \ar[dr]& & 5 \ar[dl] \ar[dr] \\
				2 & & 4 & & 6 \\
		}}$$
		and 
		$$\operatorname{low}(M):\vcenter{  \xymatrixrowsep{0.1cm} \xymatrixcolsep{0.5cm} 
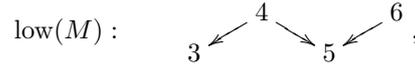
\xymatrix{
				& & 4 \ar[rd] \ar[dl] & & 6 \ar[dl] \\
				& 3 & & 5
		}},$$
		respectively.
	\end{example}
	\section{Self-orthogonal indecomposable modules}
	Having described the indecomposable $\Lambda_n$-modules, the next step towards describing the generalized tilting $\Lambda_n$-modules is determining which indecomposable $\Lambda_n$-modules are self-orthogonal, as these will be candidates for inclusion in generalized tilting modules. For details on generalized tilting modules, we refer to Subsection~\ref{subsection: generalized tilting modules}. Throughout the following sections we will make frequent use of various dimension shifting arguments. We record the most common one in the following lemma. 
	
	\begin{lemma}\label{lemma:dimension shift argument}
		Let $M$ and $N$ be finite-dimensional $\Lambda_n$-modules. 	Consider the following two short exact sequences, where $K$ is the kernel of the projective cover $P\tto M$, and $C$ is the cokernel of the injective envelope $N\hookrightarrow I$:
		$$K\hookrightarrow P \tto M, \quad N\hookrightarrow I \tto C.$$ 
		Then,
		\begin{enumerate}[(i)]
			\item
			$\dim\operatorname{Ext}_{\Lambda_n}^1(M,N)=\dim\operatorname{Hom}_{\Lambda_n}(M,N)-\dim\operatorname{Hom}_{\Lambda_n}(P,N)+\dim\operatorname{Hom}_{\Lambda_n}(K,N)$;
			\item
			$\operatorname{Ext}_{\Lambda_n}^k(M,N)\cong \operatorname{Ext}_{\Lambda_n}^{k-1}(K,N)\cong \operatorname{Ext}_{\Lambda_n}^{k-1}(M,C)$, for all $k\geq 2$;
			\item
			$\operatorname{Ext}_{\Lambda_n}^k(M,N)\cong \operatorname{Ext}_{\Lambda_n}^{k-2}(K,C)$, for all $k\geq 3$.
		\end{enumerate} 
	\end{lemma}
	\begin{proof}
		By applying $\operatorname{Hom}_{\Lambda_n}(\blank,N)$ to the first sequence, we get the long exact sequence 
		\begin{align*}
			0&\rightarrow \operatorname{Hom}_{\Lambda_n}(M,N)\rightarrow \operatorname{Hom}_{\Lambda_n}(P,N)\rightarrow \operatorname{Hom}_{\Lambda_n}(K,N)  \\
			&\rightarrow \operatorname{Ext}_{\Lambda_n}^1(M,N) \rightarrow \operatorname{Ext}_{\Lambda_n}^1(P,N) \rightarrow \operatorname{Ext}_{\Lambda_n}^1(K,N) \\
			&\rightarrow \operatorname{Ext}_{\Lambda_n}^2(M,N) \rightarrow \operatorname{Ext}_{\Lambda_n}^2(P,N) \rightarrow \operatorname{Ext}_{\Lambda_n}^2(K,N) \rightarrow \cdots 
		\end{align*}
		and, since $\operatorname{Ext}_{\Lambda_n}^k(P,N)=0$, for all $k\geq 1$, we have $$\dim\operatorname{Ext}_{\Lambda_n}^1(M,N)=\dim\operatorname{Hom}_{\Lambda_n}(M,N)-\dim\operatorname{Hom}_{\Lambda_n}(P,N)+\dim\operatorname{Hom}_{\Lambda_n}(K,N)$$
		and $$\operatorname{Ext}_{\Lambda_n}^k(M,N)\cong \operatorname{Ext}_{\Lambda_n}^{k-1}(K,N),$$ for all $k\geq 2$. 
		If we instead apply $\operatorname{Hom}_{\Lambda_n}(M,\blank)$ to the second sequence, a similar argument implies that
		$$\operatorname{Ext}_{\Lambda_n}^k(M,N)\cong \operatorname{Ext}_{\Lambda_n}^{k-1}(M,C),$$ for all $k\geq 2$. By combining these results, we obtain $$\operatorname{Ext}_{\Lambda_n}^k(M,N)\cong\operatorname{Ext}_{\Lambda_n}^{k-1}(K,N)\cong \operatorname{Ext}_{\Lambda_n}^{k-2}(K,C),$$ for all $k\geq 3$.
	\end{proof}
	
	\begin{lemma}\label{lemma:which Omega(i,j,k) are in F(Delta)}
		The module $\Omega(i,j,k)$ is contained in $\mathcal{F}(\Delta)$ if the following conditions are met.
		\begin{enumerate}[(i)] 
			\item If $i\neq 1$ and $j\neq 1$, then $\Omega(i,j,k)\in \mathcal{F}(\Delta)$ if and only if
			$i\equiv k \mod 2 \quad \text{and}\quad j\not\equiv k \mod 2$.
			\item If $i=1$ and $j\neq 1$, then $\Omega(i,j,k)\in \mathcal{F}(\Delta)$ if and only if $j\not\equiv k \mod 2$.
			\item If $i\neq 1$ and $j=1$, then $\Omega(i,j,k)\in \mathcal{F}(\Delta)$ if and only if $i \equiv k\mod 2$.
			\item If $i=j=1$, then $\Omega(i,j,k)\in\mathcal{F}(\Delta)$, for all $1\leq k\leq n$.
		\end{enumerate}
	\end{lemma}
	\begin{proof}
		We draw the module $\Omega(2,3,6)$:
		
		$$ \xymatrixrowsep{0.1cm} \xymatrixcolsep{0.5cm} 
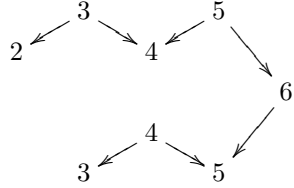
\xymatrix{
			& 3 \ar[dl] \ar[dr]& & 5 \ar[dl] \ar[ddr] \\
			2 & & 4 & & \\
			& & & & 6 \ar[ddl]\\
			& & 4 \ar[rd] \ar[dl] \\
			& 3 & & 5
		}$$
		Here it is easy to see the standard filtration, because the standard modules, pictorially, look like
		$$\Delta(k):\vcenter{\xymatrixrowsep{0.5cm}\xymatrixcolsep{0.1cm}
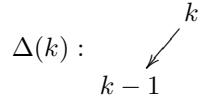
\xymatrix{
				&	k \ar[ld]\\ k-1	
		}}$$
		In this case, the subquotients of the standard filtration would be $\Delta(6), \Delta(5), \Delta(4),\Delta(3)$. In general, for a module $\Omega(i,j,k)$ with $i\neq 1$, $j\neq 1$, $i\equiv k\mod 2$ and $j\not\equiv k\mod 2$, the subquotients would be
		$$\Delta(i+1),\Delta(i+3),\dots, \Delta(k-1), \Delta(j+1),\Delta(j+3),\dots, \Delta(k).$$
		
		Let us instead draw the module $\Omega(2,4,6)$:
		$$ \xymatrixrowsep{0.1cm} \xymatrixcolsep{0.5cm} 
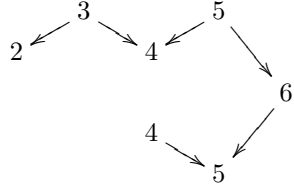
\xymatrix{
			& 3 \ar[dl] \ar[dr]& & 5 \ar[dl] \ar[ddr] \\
			2 & & 4 & & \\
			& & & & 6 \ar[ddl]\\
			& & 4 \ar[dr] \\
			& & &  5
		}$$
		Here, we see that there is no way to remedy the composition factor $L(4)$ contained in the top of $\Omega(2,4,6)$, preventing a standard filtration. The rest is similar.
	\end{proof}

	\begin{lemma}\label{lemma:which Omega(i,j,k) are in F(Nabla)}
		The module $\Omega(i,j,k)$ is contained in $\mathcal{F}(\nabla)$ if the following conditions are met.
		\begin{enumerate}[(i)]
			\item If $i\neq 1$ and $j\neq 1$, then $\Omega(i,j,k)\in \mathcal{F}(\nabla)$ if and only if
			$i\not\equiv k \mod 2 \quad \text{and}\quad j\equiv k \mod 2$.
			\item If $i=1$ and $j\neq 1$, then $\Omega(i,j,k)\in \mathcal{F}(\nabla)$ if and only if $j\equiv k \mod 2$.
			\item If $i\neq 1$ and $j=1$, then $\Omega(i,j,k)\in \mathcal{F}(\nabla)$ if and only if $i \not\equiv k\mod 2$.
			\item If $i=j=1$, then $\Omega(i,j,k)\in\mathcal{F}(\nabla)$, for all $1\leq k\leq n$.
		\end{enumerate}
	\end{lemma}
	\begin{proof}
		This follows from Lemma \ref{lemma:which Omega(i,j,k) are in F(Delta)} by using the simple-preserving duality.
	\end{proof}

	\begin{corollary}\label{corollary:characteristic tilting module}
		For all $1\leq k\leq n$, we have $T(k)=\Omega(1,1,k)$, where $T(k)$ denotes the $k$th indecomposable summand of the characteristic tilting module.
	\end{corollary}
	\begin{proof}
		It follows from Lemma \ref{lemma:which Omega(i,j,k) are in F(Delta)} and Lemma \ref{lemma:which Omega(i,j,k) are in F(Nabla)} that $\Omega(1,1,k)\in \mathcal{F}(\Delta)\cap \mathcal{F}(\nabla)$ for every $k$. The cokernel of the inclusion $\Delta(k)\hookrightarrow \Omega(1,1,k)$ is equal to $\Omega(k-1,1,k-1)\oplus \Omega(k-2,1,k-2)$. Since both $\Omega(k-1,1,k-1)$ and $\Omega(k-2,1,k-2)$ have a standard filtration by Lemma \ref{lemma:which Omega(i,j,k) are in F(Delta)}, so does their direct sum. It is clear that this standard filtration only has subquotients $\Delta(j)$ with $j<k$. This implies that $\Omega(1,1,k)=T(k)$.
	\end{proof}
	
	\begin{lemma}\label{lemma:projective covers and kernels}
		Let $M=\Omega(i,j,k)$ and let $K$ be the kernel of the projective cover $P\tto M$.
		\begin{enumerate}[(i)]
			\item If $k=i$ then $P\cong P_{(j, k]}$ and if $k>i$ then $P\cong P_{(i, k-1]} \oplus P_{(j, k-2]}$.
			\item The form of the module $K$ is given by the following table.
			
			\begin{center}
			\renewcommand{\arraystretch}{1.2}
				\begin{tabular}{>{\centering\arraybackslash}m{2.1cm}|>{\centering\arraybackslash}m{2.1cm}|>{\centering\arraybackslash}m{2.1cm}|p{7cm}}
					$i$ & $j$ & $k$ & $K$  \\ \hline  
					$1$ & $1$ & $1$ & $\Delta(2)$ \\ \hline 
					$1$ & $1$ & $k>1$ & $\Delta_{(2, k-2]}\oplus \Delta_{(2, k-1]}$ \\ 	\hline 
					$1$ & $k$ & $k>1$ & $\Delta_{(2, k]}\oplus L(k-1)$ \\
					\hline 
					$1$ & $k>j> 1$, $j\not\equiv k \mod 2$ & $k>j$ & $\Delta_{(2, k-2]}\oplus \Delta_{[j+2, k-1]}$ \\
					\hline 
					$1$ & $k>j> 1$, $j\equiv k \mod 2$ & $k>j$ & $\Delta_{(2, k-2]}\oplus \Delta_{[j+1, k-1]} \oplus L(j-1)$ \\
					\hline 
					$n$ & $1$ & $n$ & $\Delta_{(2, n-1]}$ \\
					\hline 
					$n$ & $n$ & $n$ & $L(n-1)$ \\ 
					\hline 
					$n$ & $n>j>1$, $j\not\equiv k \mod 2$ & $n$ & $\Delta_{[j+2, n-1]}$ \\
					\hline
					$n$ & $n>j>1$, $j\equiv k \mod 2$ & $n$ & $\Delta_{[j+1, n-1]}\oplus L(j-1)$ \\
					\hline
					$k$ & $1$ & $n>k>1$ & $\Delta_{(2, k+1]}$ \\
					\hline
					$k$ & $k$ & $n>k>1$ & $\Delta(k+1)\oplus L(k-1)$ \\
					\hline
					$k$ & $k>j>1$, $j\not\equiv k \mod 2$  & $n>k>j$ & $\Delta_{[j+2, k+1]}$ \\
					\hline
					$k$ & $k>j>1$, $j\equiv k \mod 2$ & $n>k>j$ & $\Delta_{[j+1, k+1]}\oplus L(j-1)$ \\
					\hline
					$k>i> 1$, $i\equiv k \mod 2$  & $1$ & $k>i$ & $\Delta_{[i+2, k-2]}\oplus \Delta_{(2, k-1]}$ \\
					\hline
					$k>i> 1$, $i\not\equiv k \mod 2$ & $1$ & $k>i$ & $\Delta_{[i+1, k-2]}\oplus \Delta_{(2, k-1]} \oplus L(i-1)$ \\
					\hline
					$k>i>1$, $i\equiv k \mod 2$ & $k$ & $k>i$ & $\Delta_{[i+2, k]} \oplus L(k-1)$ \\
					\hline
					$k>i>1$, $i\not\equiv k \mod 2$ & $k$ & $k>i$ & $\Delta_{[i+1, k]}\oplus L(i-1)\oplus L(k-1)$ \\
					\hline
					$k>i> 1$, $i\equiv k \mod 2$ & $k>j> 1$, $j\not\equiv k \mod 2$   & $k>i,j$ & $\Delta_{[i+2, k-2]}\oplus \Delta_{[j+1, k-1]}$ \\
					\hline
					$k>i> 1$, $i\not\equiv k \mod 2$ & $k>j> 1$, $j\not\equiv k \mod 2$  & $k>i,j$ & $\Delta_{[i+1, k-2]}\oplus \Delta_{[j+2, k-1]} \oplus L(i-1)$ \\
					 \hline
					$k>i> 1$, $i\equiv k \mod 2$ & $k>j> 1$,  $j\equiv k \mod 2$ & $k>i,j$ & $\Delta_{[i+2, k-2]}\oplus \Delta_{[j+1, k-1]} \oplus L(j-1)$ \\
					\hline
					$k>i> 1$, $i\not\equiv k \mod 2$ & $k>j> 1$,  $j\equiv k \mod 2$ & $k>i,j$ & $\Delta_{[i+1, k-2]}\oplus \Delta_{[j+1, k-1]} \oplus L(i-1)\oplus L(j-1)$
				\end{tabular}
			\end{center}
		\end{enumerate}
		
	\end{lemma}
	\begin{proof}
		We again consider the example $M=\Omega(2,4,6)$:
		
		$$ \xymatrixrowsep{0.1cm} \xymatrixcolsep{0.5cm} 
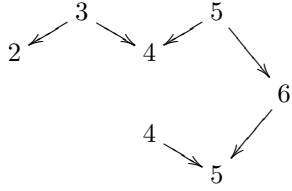
\xymatrix{
			& 3 \ar[dl] \ar[dr]& & 5 \ar[dl] \ar[ddr] \\
			2 & & 4 & & \\
			& & & & 6 \ar[ddl]\\
			& & 4 \ar[dr] \\
			& & &  5
		}$$
		We have $\operatorname{top} (M)=L(3)\oplus L(4)\oplus L(5)$ which gives a projective cover $P=P(3)\oplus P(4)\oplus P(5)$. Therefore, we have a surjection $P\tto M$, whose kernel is equal to $\Delta(4)\oplus \Delta(5)\oplus L(3)$. The remaining cases are easily ascertained by drawing the Loewy diagrams of the appropriate modules.
	\end{proof}
	\begin{lemma}\label{lemma:M not in F(delta) iff kernel has simple direct summand}
		The kernel of the projective cover of $M$ has a simple direct summand if and only if $M\notin \mathcal{F}(\Delta)$.
	\end{lemma}
	\begin{proof}
		This is easily checked by observing that the combinations of $i$, $j$ and $k$ which yield a simple direct summand of $K$, according to Lemma \ref{lemma:projective covers and kernels}, are exactly those for which $M$ does not have a $\Delta$-filtration, according to Lemma \ref{lemma:which Omega(i,j,k) are in F(Delta)}.
	\end{proof}
	
	\begin{lemma}\label{lemma:M not in F(nabla) iff cokernel has simple direct summand}
		The cokernel of the injective envelope of $M$ has a simple direct summand if and only if $M\notin \mathcal{F}(\nabla)$.
	\end{lemma}
	\begin{proof}
		This follows from Lemma \ref{lemma:M not in F(delta) iff kernel has simple direct summand} by using the simple-preserving duality.
	\end{proof}

	\begin{lemma}\label{lemma:projective resolution of simple modules}
		Let $Q_\bullet$ be the projective resolution of $L(i)$. The terms $Q_m$ of $Q_\bullet$ are given by
		$$Q_m=\begin{cases}
			P_{[i-m,i+m]}, &  \text{if } m<i \text{ and } m \leq n-i; \\
			P_{[i-m,n)}, & \text{if } n-i<m<i;\\
			P_{[m-i+2,i+m]}, & \text{if } i\leq m \leq n-i; \\
			P_{[m-i+2,n)}, & \text{if } i\leq m \text{ and } n-i < m.
		\end{cases}$$
	\end{lemma}
	\begin{proof}
		This follows from repeated application of Lemma \ref{lemma:projective covers and kernels}.
	\end{proof}
	
	\begin{lemma}\label{lemma:Ext between simples is non-zero}
		Let $i$ and $j$ be such that not both are equal to 1. Then $\operatorname{Ext}_{\Lambda_n}^m(L(i),L(j))\neq 0$ for
		$$m=\begin{cases}
			|i-j|, & \text{if }i\neq j;\\
			2, & \text{if }i=j.
		\end{cases}$$
	\end{lemma}
	\begin{proof}
		Assume that $i<j$.
		By Lemma \ref{lemma:projective resolution of simple modules}, the module $P(j)$ appears at position $j-i$ of the projective resolution of $L(i)$, yielding a non-zero extension.  The case $j<i$ follows by using the simple-preserving duality. For $i=j>1$, the same lemma tells us that $P(i)$ appears in the second position of the projective resolution of $L(i)$, which proves the claim.
	\end{proof}

	\begin{proposition}
		If an indecomposable $\Lambda_{n}$-module $M$ has neither a standard filtration nor a costandard filtration, then $M$ is not self-orthogonal. 
	\end{proposition}
	\begin{proof}
		If $M=\Omega(i,j,k)$ is not simple (this case was covered in Lemma \ref{lemma:Ext between simples is non-zero}) and has neither a standard nor a costandard filtration, then $i,j> 1$ and either $i \equiv k \mod 2$ and $j\equiv k \mod 2$, or $i\not\equiv k \mod 2$ and $j\not\equiv k \mod 2$. This follows from Lemma \ref{lemma:which Omega(i,j,k) are in F(Delta)} and Lemma \ref{lemma:which Omega(i,j,k) are in F(Nabla)}.
		
		Let $K$ denote the kernel of the projective cover $P\tto M$ and $C$ the cokernel of the injective envelope $M\hookrightarrow I$. 
		By Lemma \ref{lemma:M not in F(delta) iff kernel has simple direct summand} and Lemma \ref{lemma:M not in F(nabla) iff cokernel has simple direct summand} both $K$ and $C$ have a simple direct summand. 
		\begin{enumerate}[$\bullet$]
			\item
			If $i \equiv k \mod 2$ and $j\equiv k \mod 2$, then $L(j-1)$ is a direct summand of $K$ and $L(i-1)$ is a direct summand of $C$.
			\item
			If $i\not\equiv k \mod 2$ and $j\not\equiv k \mod 2$, then $L(i-1)$ is a direct summand of $K$ and $L(j-1)$ is a direct summand of $C$. 
		\end{enumerate}
		Unless both $i$ and $j$ are equal to $2$, using Lemma \ref{lemma:dimension shift argument} together with Lemma \ref{lemma:Ext between simples is non-zero}, we get $$\operatorname{Ext}_{\Lambda_n}^m(M,M)\cong \operatorname{Ext}_{\Lambda_n}^{m-2}(K,C)\neq 0,$$ for some $m>2$. 
		
		If $i=j=2$, then $M=\Omega(2,2,k)$ and the kernel $K$ of the projective cover is equal to $\bigoplus_{x=3}^{k-1}\Delta(x)\oplus L(1)$. The beginning of the projective resolution of $L(1)$ looks as follows: 
		\[
		\xymatrix{
			\cdots \ar[r] & P(3) \ar[rr]^{d_2} \ar[dr] && P(2) \ar[rr]^{d_1} \ar[dr] && P(1) \ar[r] & L(1) \\
			&& \Delta(3) \ar[ur] && \Delta(2) \ar[ur]
		}\]
		Since $L(2)$ is a submodule of $M$ and there is no homomorphism from $P(1)$ to $M$, we have a non-zero extension of degree 1 from $L(1)$ to $M$. But this implies that there is a non-zero extension of degree two from $M$ to itself, as $L(1)$ is a direct summand of $K$ and, by Lemma \ref{lemma:dimension shift argument}, we have $\operatorname{Ext}_{\Lambda_n}^2(M,M)\cong \operatorname{Ext}_{\Lambda_n}^1(K,M)$. This proves the claim. 
	\end{proof}

	\begin{lemma} \label{lemma:Ext from simples to F(nabla) is zero}
		Let $M=\Omega(i,j,k)\in\mathcal{F}(\nabla)$. Then $\operatorname{Ext}_{\Lambda_n}^m(L(x),M)=0$, for all $1\leq x<\min(i,j)$ and $m\geq 0$.
	\end{lemma}
	\begin{proof}
		It is clear that $\operatorname{Hom}_{\Lambda_n}(L(x),M)=0$, for all $1\leq x<\min(i,j)$, since $M$ does not have $L(x)$ as a composition factor for such $x$. It is also clear that $\operatorname{Ext}_{\Lambda_n}^m(L(1),M)=0$, for $m>0$, since $L(1)\in \mathcal{F}(\Delta)$ and $M\in\mathcal{F}(\nabla)$. This proves the claim for $x=1$. 
		
		Consider $x$ such that $2\leq x <\min(i,j)$.
		In this case, we have the short exact sequence $$L(x-1)\oplus \Delta(x+1)\hookrightarrow P(x)\tto L(x),$$ and by Lemma~\ref{lemma:dimension shift argument} we have the following equality:
		\begin{displaymath}
			\dim\operatorname{Ext}_{\Lambda_n}^1(L(x),M)=\dim\operatorname{Hom}_{\Lambda_n}(L(x),M)-\dim\operatorname{Hom}_{\Lambda_n}(P(x),M) 
			+\dim\operatorname{Hom}_{\Lambda_n}(L(x-1)\oplus \Delta(x+1),M).		    
		\end{displaymath}
		
		Since neither $L(x)$ nor $L(x-1)$ are composition factors of $M$ this equality reduces to
		$$\dim\operatorname{Ext}_{\Lambda_n}^1(L(x),M)=\dim\operatorname{Hom}_{\Lambda_n}(\Delta(x+1),M).$$
		If $2\leq x<\min(i,j)-1$, then clearly $$\operatorname{Hom}_{\Lambda_n}(\Delta(x+1),M)=0,$$ since $L(x+1)$ does not occur as composition factors in $M$. 
		When $x=\min(i,j)-1$ we also have $\operatorname{Hom}_{\Lambda_n}(\Delta(x+1),M)=0$. To see this, note that any non-zero homomorphism $f:\Delta(x+1)\to M$ annihilates $\operatorname{rad}(\Delta(x+1))=L(x)$ since $L(x)$ does not occur as a composition factor in $M$. This implies that the image of $f$ is isomorphic to $L(x+1)$. But $M$ has no such submodule since the (unique) composition factor $L(x+1)$ is contained in the top of $M$, a contradiction. It follows that $\operatorname{Ext}_{\Lambda_n}^1(L(x),M)=0$ for all $x<\min(i,j)$.
		
		It remains to show that $\operatorname{Ext}_{\Lambda_{n}}^m(L(x),M)=0$ for all $m>1$.
		By Lemma~\ref{lemma:dimension shift argument} and the short exact sequence above, together with the fact that $M\in \mathcal{F}(\nabla)$, it follows that 
		$$\operatorname{Ext}_{\Lambda_n}^m(L(x),M)\cong \operatorname{Ext}_{\Lambda_n}^{m-1}(L(x-1)\oplus \Delta(x+1),M)\cong \operatorname{Ext}_{\Lambda_n}^{m-1}(L(x-1),M).$$

		If $m< x$, then by repeated use of the isomorphism above we have $$\operatorname{Ext}_{\Lambda_n}^m(L(x),M)\cong \operatorname{Ext}_{\Lambda_n}^{1}(L(x-m+1),M)=0,$$ by the previous case. If $m\geq x$, then $$\operatorname{Ext}_{\Lambda_n}^m(L(x),M)\cong \operatorname{Ext}_{\Lambda_n}^{m-x+1}(L(1),M)=0,$$
		since $L(1)\in\mathcal{F}(\Delta)$ and $M\in\mathcal{F}(\nabla)$.
	\end{proof}
	
	\begin{proposition}\label{proposition:characterization of self-orthogonal modules}
		Let $\Omega(i,j,k)\in \mathcal{F}(\Delta)\cup \mathcal{F}(\nabla)$ be such that $\Omega(i,j,k)\notin \mathcal{F}(\Delta)\cap \mathcal{F}(\nabla)$. Then $\Omega(i,j,k)$ is self-orthogonal if and only if $|i-j|=1$.
	\end{proposition}
	\begin{proof}
		We will show that the claim holds for $M=\Omega(i,j,k)$ such that $M\in\mathcal{F}(\nabla)$, but $M\notin \mathcal{F}(\Delta)$. It will then follow, by applying the simple-preserving duality, that the claim also holds for $M$ such that $M\in\mathcal{F}(\Delta)$, but $M\notin\mathcal{F}(\nabla)$. Note that if $i=j=1$, then $M\in \mathcal{F}(\Delta)\cap \mathcal{F}(\nabla)$, and if $i=j>1$, then $M\notin \mathcal{F}(\Delta)\cup \mathcal{F}(\nabla)$. Thus, we must have $|i-j|\geq 1$.
		
		Assume that $|i-j|=1$. 
		Consider the surjection $T(k) \tto M$ (which is unique up to a scalar) and denote the kernel of this projection by $K$. If $i,j>1$, then the kernel $K$ can be written as a direct sum of two indecomposable modules $U\oplus L$, where $U=\Omega(1,i-1,i-1)$ and $L=\Omega(1,j-1,j-1)$. If $i=1$ and $j=2$, or $i=2$ and $j=1$, then the kernel $K$ is equal to $L(1)$. This gives us a short exact sequence $$K \hookrightarrow T(k) \tto M.$$
		By applying $\operatorname{Hom}_{\Lambda_n}(\blank,M)$, we get a long exact sequence
		\begin{align*}
			0&\rightarrow \operatorname{Hom}_{\Lambda_n}(M,M)\rightarrow \operatorname{Hom}_{\Lambda_n}(T(k),M)\rightarrow \operatorname{Hom}_{\Lambda_n}(K,M)  \\
			&\rightarrow \operatorname{Ext}_{\Lambda_n}^1(M,M) \rightarrow \operatorname{Ext}_{\Lambda_n}^1(T(k),M) \rightarrow \operatorname{Ext}_{\Lambda_n}^1(K,M) \\
			&\rightarrow \operatorname{Ext}_{\Lambda_n}^2(M,M) \rightarrow \operatorname{Ext}_{\Lambda_n}^2(T(k),M) \rightarrow \operatorname{Ext}_{\Lambda_n}^2(K,M) \rightarrow \cdots 
		\end{align*}
		Since $M\in\mathcal{F}(\nabla)$ and $T(k)\in\mathcal{F}(\Delta)$, we have $\operatorname{Ext}_{\Lambda_n}^m(T(k),M)=0$ for all $m>0$. This implies that \[\operatorname{Ext}_{\Lambda_n}^m(M,M)\cong \operatorname{Ext}_{\Lambda_n}^{m-1}(K,M),\] for all $m\geq 2$. 
		
		Note that the only common composition factor of $K$ and $M$ is $L(x)$, where $x=\min(i,j)$. However, $L(x)$ is not a submodule of $M$, which means that there are no non-zero homomorphisms from $K$ to $M$. Together with the fact that $\operatorname{Ext}^1_{\Lambda_{n}}(T(k),M)=0$, this implies $\operatorname{Ext}_{\Lambda_n}^1(M,M)=0$.

		If $i=1$ and $j=2$, or $i=2$ and $j=1$, then $\operatorname{Ext}_{\Lambda_{n}}^m(M,M)\cong\operatorname{Ext}_{\Lambda_n}^{m-1}(K,M)=0$, for all $m\geq 2$, since $K=L(1)\in\mathcal{F}(\Delta)$. 
		Now assume that $i,j>1$. By Lemma \ref{lemma:projective covers and kernels}, the kernel $J$ of the projective cover of $K=U\oplus L$ is equal to $\Delta_{(2,i-3]}\oplus \Delta_{(2,j-3]}\oplus L(i-2)\oplus L(j-2)$, where $L(x)$ is interpreted as 0 if $x<1$. Using Lemma \ref{lemma:Ext from simples to F(nabla) is zero} and the fact that $M\in \mathcal{F}(\nabla)$, this implies that $\operatorname{Ext}_{\Lambda_n}^{m-2}(J,M)=0.$
		It now follows that
		$$\operatorname{Ext}_{\Lambda_n}^m(M,M)\cong \operatorname{Ext}_{\Lambda_n}^{m-1}(K,M)\cong \operatorname{Ext}_{\Lambda_n}^{m-2}(J,M)=0,$$
		for all $m\geq 2$. Note that the second isomorphism is obtained by applying Lemma \ref{lemma:dimension shift argument}. This proves the claim that, if $|i-j|=1$, then $\Omega(i,j,k)$ is self-orthogonal.

		Assume that $|i-j|>1$.
		The projective cover of $M$ will be $P_{[i,k-1]}\oplus P_{(j,k-2]}$ and the corresponding kernel will be $\Delta_{(i+1,k-2]}\oplus \Delta_{(j+1,k-1]}\oplus L(i-1)\oplus L(j-1)$, where $L(x)$ is interpreted as 0 if $x<1$. This, together with Lemma \ref{lemma:dimension shift argument} and the fact that $M\in\mathcal{F}(\nabla)$, implies that \[\operatorname{Ext}_{\Lambda_n}^m(M,M)\cong \operatorname{Ext}_{\Lambda_n}^{m-1}(L(i-1),M)\oplus \operatorname{Ext}_{\Lambda_n}^{m-1}(L(j-1),M),\]
		for all $m\geq 1$.
		
		We will prove that either $\operatorname{Ext}_{\Lambda_n}^{1}(L(i-1),M)\neq 0$, or $\operatorname{Ext}_{\Lambda_n}^{1}(L(j-1),M)\neq 0$, depending on whether $j\leq i-2$ or $i\leq j-2$. We will consider the case when $j\leq i-2$. The case $i\leq j-2$ is proven using the exact same arguments, but with $i$ replaced by $j$.
		
		The beginning of the projective resolution of $L(i-1)$ looks as follows:
		\[
		\xymatrix{
			\cdots \ar[r]^-{d_2} & P(i-2)\oplus P(i) \ar[rr]^-{d_1} \ar[dr] && P(i-1) \ar[r]^-{d_0} & L(i-1) \\
			&& L(i-2)\oplus \Delta(i) \ar[ur]
		}\]
		Since $L(i-2)$ is a submodule of $M$, there is a homomorphism $f:P(i-2)\oplus P(i)\to M$ such that the image of $f$ is isomorphic to $L(i-2)$. Furthermore, the composition $f\circ d_2$ is equal to the zero homomorphism. However, we cannot have $f=g\circ d_1$ for any homomorphism $g:P(i-1)\to M$. Indeed, the kernel of the unique (up to scalar) non-zero map from $P(i-1)$ to $M$ is equal to $L(i-1)$. The image of $g\circ d_1$ therefore contains the submodule $L(i)$. Since the image of $f$ is isomorphic to $L(i-2)$ the two maps cannot be equal. This means that $\operatorname{Ext}_{\Lambda_n}^1(L(i-1),M)\neq 0$, which implies that $\operatorname{Ext}_{\Lambda_n}^2(M,M)\neq 0$.
	\end{proof}
	
	\section{Generalized tilting modules}
	\subsection{Background of generalized tilting modules} \label{subsection: generalized tilting modules}
	The main goal of this article is to classify all generalized tilting $\Lambda_n$-modules. Classical tilting modules were first introduced by \cite{BB80, HR82} and were later generalized by Miyashita in \cite{Mi86}. Recall that for a $\Lambda_n$-module $M$, $\operatorname{add}M$ denotes the full subcategory of $\Lambda_n\operatorname{-mod}$ consisting of direct summands of finite direct sums of $M$.
	
	\begin{definition}\cite{Mi86}
		Let $\Lambda$ be an algebra and let $T$ be a $\Lambda$-module. Then, $T$ is called a \emph{generalized tilting module} if
		\begin{enumerate}
			\item[(T1)] $T$ has finite projective dimension;
			\item[(T2)] $\operatorname{Ext}_{\Lambda}^m(T,T)=0$ for all $m>0$;
			\item[(T3)] there is an exact sequence
			$$\xymatrix{
				0\ar[r]& \Lambda \ar[r] & Q_0\ar[r] & Q_1 \ar[r] & \dots \ar[r]& Q_r\ar[r] & 0
			}$$
			such that $Q_i \in \operatorname{add} T$ for all $0\leq i\leq r$.
		\end{enumerate}
	\end{definition}
	
	Recall that every quasi-hereditary algebra has finite global dimension, so (T1) is satisfied for every $\Lambda_n$-module.

	\begin{theorem}\cite{rickard_schofield_1989} \label{theorem:RS}
		Let $\Lambda$ be an algebra of finite representation type and let $T$ be a $\Lambda$-module satisfying the first two properties of a generalized tilting module:
		\begin{enumerate}
			\item[(T1)]
			$T$ has finite projective dimension;
			\item[(T2)]
			$\operatorname{Ext}_\Lambda^m(T,T)=0$, for all $m>0$.
		\end{enumerate}
		Then, there is  a $\Lambda$-module $S$, such that $T\oplus S$ is a generalized tilting module. 
	\end{theorem}
	\begin{corollary}\cite{rickard_schofield_1989} \label{corollary:theorem:RS}
		Let $\Lambda$ be an algebra of finite representation type and let $T$ be a $\Lambda$-module satisfying:
		\begin{enumerate}
			\item[(T1)] $T$ has finite projective dimension;
			\item[(T2)] $\operatorname{Ext}_\Lambda^m(T,T)=0$, for all $m>0$;
			\item[(T3$^\prime$)] $T$ has $n$ non-isomorphic indecomposable direct summands.
		\end{enumerate}
		Then, $T$ is a generalized tilting module.
	\end{corollary}
	
	Corollary \ref{corollary:theorem:RS} implies that to classify generalized tilting modules, it is enough to classify all collections of $n$ indecomposable self-orthogonal modules such that all extensions (of positive degree) between each pair of modules vanish.

	\subsection{Non-zero extensions between modules in $\mathcal{F}(\nabla)$ and $\mathcal{F}(\Delta)$}
	The aim of this subsection is to prove that for any self-orthogonal modules $M\in \mathcal{F}(\nabla)$ and $N\in \mathcal{F}(\Delta)$, there is a non-zero extension, from $M$ to $N$, of positive degree. This reduces the problem of classifying all generalized tilting $\Lambda_n$-modules to finding all generalized tilting modules in $\mathcal{F}(\Delta)$ and then, by using the simple-preserving duality, obtaining all generalized tilting modules in $\mathcal{F}(\nabla)$.

	\begin{proposition}\label{proposition:non-zero extensions between F(Delta) and F(nabla)}
		Let $M\in \mathcal{F}(\nabla)$ and $N\in \mathcal{F}(\Delta)$ be indecomposable $\Lambda_{n}$-modules such that $M$ and $N$ are self-orthogonal and $M,N\not\in \mathcal{F}(\Delta)\cap \mathcal{F}(\nabla)$. Then, $\operatorname{Ext}_{\Lambda_n}^m(M,N)\neq 0$, for some $m\geq 1$.
	\end{proposition}
	\begin{proof}
		By Proposition \ref{proposition:characterization of self-orthogonal modules}, $M$ and $N$ must be of one of the following forms. 
		\begin{enumerate}[$\bullet$]
			\item
			$M=\Omega(i,i+1,k)$ if $i\equiv k \mod 2$;
			\item
			$M=\Omega(i+1,i,k)$ if $i\not\equiv k \mod 2$;
			\item
			$N=\Omega(j+1,j,\ell)$ if $j\equiv \ell \mod 2$;
			\item
			$N=\Omega(j,j+1,\ell)$ if $j\not\equiv \ell \mod 2$.
		\end{enumerate}
		
		By Lemma \ref{lemma:projective covers and kernels}, the kernel $K$ of the projective cover of $M$ is equal to $\bigoplus_{x=i+1}^{k-1}\Delta(x)\oplus L(i-1)\oplus L(i)$, where $L(x)$ is interpreted as zero if $x<1$.
		By dualizing Lemma \ref{lemma:projective covers and kernels}, the cokernel $C$ of the injective envelope of $N$ is equal to $\bigoplus_{x=j+1}^{\ell-1}\nabla(x)\oplus L(j-1)\oplus L(j)$, where $L(x)$ is interpreted as zero if $x<1$.

		Unless $i=j=1$, using Lemma \ref{lemma:dimension shift argument} together with Lemma \ref{lemma:Ext between simples is non-zero}, we get $\operatorname{Ext}_{\Lambda_n}^m(M,N)\cong\operatorname{Ext}_{\Lambda_n}^{m-2}(K,C)\neq 0$, for some $m>0$.
		Assume that $i=j=1$. Since $i=1$, as mentioned before, the kernel $K$ of the projective cover of $M$ is equal to $\bigoplus_{x=2}^{k-1}\Delta(x)\oplus L(1)$. The beginning of the projective resolution of $L(1)$ looks as follows: 
		\[
		\xymatrix{
			\cdots \ar[r] & P(3) \ar[rr]^{d_2} \ar[dr] && P(2) \ar[rr]^{d_1} \ar[dr] && P(1) \ar[r] & L(1) \\
			&& \Delta(3) \ar[ur] && \Delta(2) \ar[ur]
		}\]
		If $\ell=2$, then $N=\Delta(2)$ and we a non-split short exact sequence $$\Delta(2)\hookrightarrow P(1) \tto L(1).$$ This, together with Lemma~\ref{lemma:dimension shift argument}, implies that there is a non-zero extension of degree two from $M$ to $N$.
		
		If $\ell>2$, there is a homomorphism $f:P(2)\to N$, whose image is isomorphic to $L(2)$ and annihilates $\operatorname{rad}P(2)$. Since the image of the differential $d_2$ is contained in $\operatorname{rad}P(2)$, we have $f\circ d_2=0$. However, the unique (up to scalar) homomorphism $g:P(1)\to N$ is such that $L(2)$ does not occur as a composition factor in the image, so that we must have $f\neq g\circ d_1$. Therefore, we have a non-zero extension of degree one from $L(1)$ to $N$. But this, together with Lemma~\ref{lemma:dimension shift argument}, implies that there is a non-zero extension of degree two from $M$ to $N$. This proves the claim. 
	\end{proof}

	\subsection{A strict partial order on $\mathcal{F}(\Delta)$}
	We know that the only generalized tilting module contained in $\mathcal{F}(\Delta)\cap \mathcal{F}(\nabla)$ is the characteristic tilting module. Moreover, we have shown that if $M\in \mathcal{F}(\nabla)$, $N\in \mathcal{F}(\Delta)$ and $M,N\notin \mathcal{F}(\Delta)\cap \mathcal{F}(\nabla)$, then $\operatorname{Ext}_{\Lambda_n}^m(M,N)\neq 0$ for some $m\geq 1$. This implies that any generalized tilting module must be contained in either $\mathcal{F}(\Delta)$ or $\mathcal{F}(\nabla)$. Together, these statements imply that it is enough to find the basic generalized tilting modules contained in $\mathcal{F}(\Delta)$, which do not equal the characteristic tilting module. All basic generalized tilting modules will then be these modules, their duals (with respect to the simple preserving duality), and the characteristic tilting module.

	Let $\mathcal{D}_n$ denote the set of indecomposable self-orthogonal modules in $\mathcal{F}(\Delta)$.
	Next, we define the relation $\prec_E$ on $\mathcal{D}_n$ given by $M\prec_E N$ if and only if $\operatorname{Ext}^m(M,N)\neq 0$ for some $m\geq 1$.
	Note that it is not clear from the definition whether this relation is transitive. 
	
	Let $Q_n$ be the set of pairs $(i,k)$, with $1\leq i\leq k\leq n$. We want these pairs to encode the indecomposable self-orthogonal modules contained in $\mathcal{F}(\Delta)$. Let $M(i,k)$ be the module
	
	$$M(i,k)=\begin{cases}
		\Omega(1,1,k), & \text{if } i=1;\\
		\Omega(i,i-1,k), & \text{if } i>1 \text{ and } i\equiv k\mod 2;\\
		\Omega(i-1,i,k), & \text{if } i>1 \text{ and } i\not\equiv k\mod 2.
	\end{cases}$$
	This defines a bijection $\varphi:Q_n \rightarrow \mathcal{D}_n$ given by $(i,k)\mapsto M(i,k)$.
	
	We define the following relation on $Q_n$:
	$$(i,i)\prec_0 (i+1,\ell), \text{ for } \ell=i+1,\dots,n.$$
	For $k=1,2,\dots$ we have the additional relations
	\begin{align*}
		(i,i+2k)&\prec_0 (i+1,i+1+2\ell), &\text{ for } &\ell=0,1,\dots,k-1; \\
		(i,i+2k)&\prec_0 (i+1,\ell), &\text{ for } &\ell=i+2k+1, i+2k+2, \dots, n;\\
		(i,i+2k+1)&\prec_0 (i+1,i+1+2\ell), &\text{ for } &\ell=0,1,\dots, k-1.
	\end{align*}
	Now define $\prec$ to be the transitive closure of $\prec_0$. This defines a strict partial order on $Q_n$. By a \emph{strict partial order} on a set, we mean a relation which is transitive, asymmetric and irreflexive. The set $Q_n$ is naturally graded by $\deg(i,k)=i$, and from the definition it is clear that this grading, together with the strict partial order $\prec$, makes $Q_n$ into a graded poset.
	
	Note that, since we have the relations $(i,k)\prec_0 (i+1,i+1)$, if $k\neq i+1$, and $(i+1,i+1)\prec_0 (i+2,\ell)$, for all $\ell\geq j+2$, it follows that, if $k\neq i+1$, then $$(i,k)\prec (j,\ell),$$
	for every $j\geq i+2$ and every $\ell\geq j$. 
	
	The aim of the next subsection is to prove the following theorem:
	\begin{theorem}\label{theorem:order isomorphism}
		The bijection $\varphi:Q_n \rightarrow \mathcal{D}_n$ defined above is an order isomorphism between $(Q_n,\prec)$ and $(\mathcal{D}_n, \prec_E)$.
	\end{theorem}
	
	Before proving the theorem, let us briefly consider its implications. Let $n=4$. We have the following Hasse diagram for the graded poset $(Q_n,\prec)$.
	
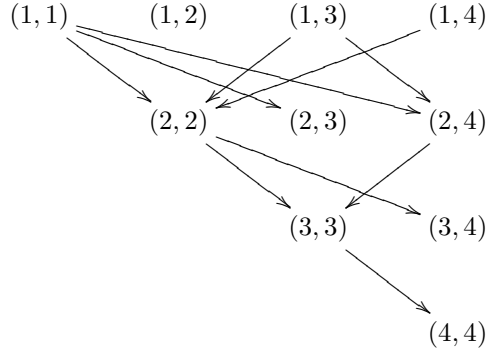
	$$\xymatrix{
		(1,1) \ar[rd] \ar[rrd] \ar[rrrd] & (1,2) & (1,3) \ar[rd]\ar[ld] & (1,4) \ar[lld]\\
		& (2,2) \ar[rd] \ar[rrd]& (2,3) & (2,4) \ar[ld]\\
		& & (3,3) \ar[rd]& (3,4)\\
		& & & (4, 4)	
	}$$
	
	We view the above picture as a graph, with the vertices $(i,k)$ corresponding to indecomposable self-orthogonal modules in $\mathcal{F}(\Delta)$. We draw an edge $(i,k)\to (j, \ell)$ if and only if $(i,k) \prec (j, \ell)$, which, as we will show, holds if and only if there is a non-zero extension between the corresponding modules. 
	
	With the theorem established, we will be able to find all generalized tilting modules over $\Lambda_n$ which are contained in $\mathcal{F}(\Delta)$ by finding all anti-chains of length $n$ in the above graph. To see this, note that an anti-chain of length $n$ corresponds exactly to a self-orthogonal module of finite projective dimension, having $n$ indecomposable summands. Such a module is a generalized tilting module by, Corollary~\ref{corollary:theorem:RS}.
	
	In the picture above, we can simply read off the anti-chains from the picture. We see that
	\begin{gather*}
		\{(1,1), (1,2), (1,3), (1,4)\},\quad \{(1,2), (1,3), (1,4), (2,3)\},\\
		\{(1,2), (2,2), (2,3), (2,4)\},\quad \{(1,2), (1,4), (2,3), (2,4)\}, \\
		\{(1,2), (2,3), (3,3), (3,4)\},\quad \{(1,2), (2,3), (2,4), (3,4)\} \\
		\text{and}\quad \{(1,2), (2,3), (3,4), (4,4)\}
	\end{gather*}
	are all possible anti-chains. The first anti-chain corresponds to the characteristic tilting module, and the last anti-chain corresponds to the module $P(1)\oplus P(2)\oplus P(3)\oplus P(4)$.
	
	\subsection{Proof of Theorem~\ref{theorem:order isomorphism}}
	Let $M,N$ be two $\Lambda_n$-modules. Consider the following two exact sequences, where $K$ is the kernel of the projective cover $P\tto M$ and $C$ is the cokernel of the injective envelope $N\hookrightarrow I$:
	$$K\hookrightarrow P \tto M, \quad N\hookrightarrow I \tto C.$$ 
	
	If $M$ is the module corresponding to $(i,k)$, where $k>i$, then, by Lemma \ref{lemma:projective covers and kernels}, we have $P=\bigoplus_{x=i}^{k-1}P(x)$ and $K=\bigoplus_{x=i+1}^{k-1}\Delta(x)$. 
	The module corresponding to $(i,i)$ is $\Delta(i)$ and, in this case, $P=P(i)$ and $K=\Delta(i+1)$.
	If $N$ is the module corresponding to $(j,\ell)$, then, by looking at $N^\star$ and using Lemma \ref{lemma:projective covers and kernels}, we find that $I=\bigoplus_{x=j-1}^{\ell-1}I(x)$ and $C=\bigoplus_{x=j}^{\ell-1}\nabla(x)\oplus L(j-2)\oplus L(j-1)$, where $L(x)$ is interpreted as zero if $x<1$. 
	
	Recall that, by Lemma \ref{lemma:dimension shift argument}, we have the following.
	
	\begin{enumerate}[(i)]
		\item
		$\dim\operatorname{Ext}_{\Lambda_n}^1(M,N)=\dim\operatorname{Hom}_{\Lambda_n}(M,N)-\dim\operatorname{Hom}_{\Lambda_n}(P,N)+\dim\operatorname{Hom}_{\Lambda_n}(K,N)$;
		\item
		$\operatorname{Ext}_{\Lambda_n}^k(M,N)\cong \operatorname{Ext}_{\Lambda_n}^{k-1}(K,N)\cong \operatorname{Ext}_{\Lambda_n}^{k-1}(M,C)$, for all $k\geq 2$;
		\item
		$\operatorname{Ext}_{\Lambda_n}^k(M,N)\cong \operatorname{Ext}_{\Lambda_n}^{k-2}(K,C)$, for all $k\geq 3$.
	\end{enumerate}

	Using the above, together with Lemma~\ref{lemma:projective covers and kernels}, we get the following equations for the dimensions of extension spaces from $M$ to $N$.
	\begin{align}\label{equation: dim Ext^1}
		\dim\operatorname{Ext}_{\Lambda_n}^1(M,N)=&\dim\operatorname{Hom}_{\Lambda_n}(M,N)-\dim\operatorname{Hom}_{\Lambda_n}(P_X,N) 
		+\dim\operatorname{Hom}_{\Lambda_n}(\Delta_Y,N), 
	\end{align}
	where $X=\{i,i+1,\dots,k-1\}$, $Y=\{i+1,i+2,\dots, k-1\}$, if $k>i$ and $X=\{i\}$, $Y=\{i+1\}$, if $k=i$.
	\begin{align}\label{equation: dim Ext^2}
		\dim\operatorname{Ext}_{\Lambda_n}^2(M,N)=&\dim\operatorname{Hom}_{\Lambda_n}(\Delta_X,N)-\dim\operatorname{Hom}_{\Lambda_n}(P_X,N) 
		+\dim\operatorname{Hom}_{\Lambda_n}(\Delta_Y,N), 
	\end{align}
	where $X=\{i+1,i+2,\dots,k-1\}$, $Y=\{i+2,i+3,\dots, k\}$, if $k>i$ and $X=\{i+1\}$, $Y=\{i+2\}$, if $k=i$.
	\begin{align}\label{equation: dim Ext^m m>=3}
		\dim\operatorname{Ext}_{\Lambda_n}^m(M,N)=\dim\operatorname{Ext}_{\Lambda_n}^{m-2}(\Delta_X,L(j-2)\oplus L(j-1)),
	\end{align}
	where $m\geq 3$, $X=\{i,i+1,\dots,k-1\}$, if $k>i$, and $X=\{i\}$, if $k=i$. Note that for the last equality we have used the fact that $\operatorname{Ext}_{\Lambda_n}^m(\Delta(x),\nabla(y))=0$, for all $m>0$.
	
	These equations indicate that we need to find the dimensions of various homomorphism spaces, as well as determining when we have a non-zero extension of positive degree from a standard module to a simple module, or not.

	\begin{lemma}\label{lemma: dim of hom-spaces}
		We have the following dimensions for the homomorphism spaces from a projective module to $M(i,k)$ and from a standard module to $M(i,k)$, respectively:
		\[
		\dim\operatorname{Hom}_{\Lambda_n}(P(x),M(i,k))=
		\begin{cases} 
			1, & \text{if } x=i-1 \lor x=k; \\
			2, & \text{if } i\leq x \leq k-1; \\
			0, & \text{if } x<i-1 \lor x>k.
		\end{cases}
		\]
		\[
		\dim\operatorname{Hom}_{\Lambda_n}(\Delta(x),M(i,k))=
		\begin{cases} 
			1, & \text{if } i-1\leq x \leq k; \\
			0, & \text{if } x<i-1 \lor x>k.
		\end{cases}
		\]
	\end{lemma}
	\begin{proof}
		Since $\dim\operatorname{Hom}_{\Lambda_n}(P(x),M(i,k))$ is equal to the number of composition factors $L(x)$ in $M(i,k)$, the first result follows immediately from the definition of $M(i,k)$. 
		If $x<i-1$, or $x>k$, then $\Delta(x)$ and $M(i,k)$ do not have any common composition factors, and there cannot exist any non-zero homomorphisms from $\Delta(x)$ to $M(i,k)$. 
		For $x=k$, we note that $\Delta(k)$ is a submodule of $M(i,k)$, and, since there is only one composition factor $L(k)$ in $M(i,k)$, it is clear that the inclusion of $\Delta(k)$ into $M(i,k)$ is the only homomorphism.
		
		Now assume that $i-1\leq x\leq k-1$. We claim that, for each such $x$, there is exactly one homomorphism from $\Delta(x)$ to $M(i,k)$ (up to a scalar), namely the homomorphism that sends the top of $\Delta(x)$ to the (unique) submodule $L(x)\subset M(i,k)$ and annihilates the radical of $\Delta(x)$. If $x\neq 1$, or equivalently, if $\Delta(x)$ is not simple, the only other possibility would be an inclusion of $\Delta(x)$ into $M(i,k)$, but for such $x$, the module $\Delta(x)$ is not a submodule of $M(i,k)$.
	\end{proof}
	
	The following lemma can be proven using the main theorem in \cite{CB89}, but for pedagogical reasons we give a more elementary proof.

	\begin{lemma}\label{lemma: dim of hom-spaces in F(Delta)}
		The space $\operatorname{Hom}_{\Lambda_n}(M(i,k),M(j,\ell))$, where $i+2\leq k$, has the following dimension, depending on $i,j,k,\ell$:
		\begin{enumerate}[$\bullet$]
			\item
			$\dim\operatorname{Hom}_{\Lambda_n}(M(i,k),M(j,\ell))=0$, if $i>\ell$ or $j>k$.
			\item $\dim \operatorname{Hom}_{\Lambda_n}(M(i,k),M(j,\ell))=1$, if $j=k$ (and $i\leq \ell$).
			\item $\dim \operatorname{Hom}_{\Lambda_n}(M(i,k),M(j,\ell))=\min(k,\ell)-\max(i,j-1)$, if $i\leq \ell$, $j<k$ and, in addition, one of the following hold: 
			\begin{enumerate}[$-$]
				\item
				$j\geq i+2$;
				\item
				$j=i+1$ and $\ell=k$;
				\item
				$j=i+1, \ell<k$ and $i\equiv \ell\mod 2$;
				\item
				$j=i+1, \ell>k$ and $i\not\equiv \ell\mod 2$.
			\end{enumerate}
			\item $\dim \operatorname{Hom}_{\Lambda_n}(M(i,k),M(j,\ell))=\min(k,\ell)-\max(i,j-1)+1$, if $i\leq \ell$, $j<k$ and, in addition, one of the following hold: 
			\begin{enumerate}[$-$]
				\item
				$j=i+1, \ell<k$ and $i\not\equiv \ell\mod 2$;
				\item
				$j=i+1, \ell>k$ and $i\equiv \ell\mod 2$;
				\item
				$j\leq i$.
			\end{enumerate}
		\end{enumerate}
	\end{lemma}
	\begin{proof}
		First we note that we have the following formulas for the top of $M(i,k)$ and the socle of $M(j,\ell)$:
		\begin{align*}
			\operatorname{top}(M(i,k))&=\medoplus_{x=i}^{k-1}L(x), \\
			\operatorname{soc}(M(j,\ell))&=
			\left\{\begin{array}{c l}
				\medoplus\limits_{x=j-1}^{\ell-1} L(x), & \text{if } j>1; \\
				\medoplus\limits_{x=1}^{\ell-1}L(x), & \text{if } j=1.
			\end{array}\right.
		\end{align*}
		Observe that, if $i>\ell$ or $j>k$, then no direct summand of $\operatorname{top}(M(i,k))$ is a composition factor of $M(j,\ell)$, so $\dim\operatorname{Hom}_{\Lambda_n}(M(i,k),M(j,\ell))=0$. If $j=k$ (and thus $i\leq \ell$), then $\dim \operatorname{Hom}_{\Lambda_n}(M(i,k),M(j,\ell))=1$, and a basis vector in $\operatorname{Hom}_{\Lambda_n}(M(i,k),M(j,\ell))$ is given by the homomorphism defined by mapping $L(k-1)= L(j-1)\subset \operatorname{top}(M(i,k))$ to the unique submodule $L(j-1)\subset \operatorname{soc}(M(j,\ell))$. 
		
		Assume throughout the rest of the proof that $i\leq \ell$ and $j< k$. Then, each composition factor $L(x)$, belonging to the top of $M(i,k)$, can be mapped to the composition factor $L(x)$, contained in the socle of $M(j,\ell)$, for each $x$ such that $\max(i,j-1)\leq x \leq \min(k-1,\ell-1)$. This means that $$\dim\operatorname{Hom}_{\Lambda_n}(M(i,k),M(j,\ell))\geq \min(k,\ell)-\max(i,j-1).$$
		We will now investigate whether or not there is an additional homomorphism  $f$ that does not belong to the subspace spanned by the homomorphisms previously mentioned. 
		To get such a homomorphism $f$, some composition factor $L(x)$, contained in the top of $M(i,k)$, must be mapped to the (unique) composition factor $L(x)$ not contained in the socle of $M(j,\ell)$. This means that either $L(x)$ is contained in the top of $M(j,\ell)$, or that $x=\ell$.

		Assume that $f$ is such a homomorphism. Since the image of $f$ is a submodule, the composition factors $L(x-1)$ and $L(x+1)$, contained in the radical of $M(j,\ell)$, must also belong to the image of $f$, assuming that they are composition factors of $M(j,\ell)$. This means that the composition factors $L(x-1)$ and $L(x+1)$, contained in the radical of $M(i,k)$, do not belong to the kernel of $f$. But the kernel is a submodule, so any composition factors having arrows to $L(x-1)$ or $L(x+1)$ in the Loewy diagram of $M(i,k)$ also do not belong to the kernel of $f$. Repeating these two arguments we will either reach a contradiction, meaning that there are no more homomorphisms, or come to the conclusion that there is (up to scalar) exactly one more homomorphism.
		
		We will investigate whether or not there exists $x$ such that a homomorphism $f$ could map $L(x)$, contained in the top of $M(i,k)$, to the (unique) composition $L(x)$ not contained in the socle of $M(j,\ell)$. This will depend on the parameters $i,j,k,\ell, x$.
		
		\begin{enumerate}[$\bullet$]
			\item Assume that $j\geq i+2$.
			Repeating the arguments above we find that one of the composition factors $L(j-1)$ or $L(j-2)$, depending on the parity of $x$ and $j$, both contained in the top of $M(i,k)$, are not contained in the kernel of $f$. This is a contradiction. Indeed, 
			$L(j-2)$ is not a composition factor of $M(j,\ell)$, so $L(j-2)$ must be contained in the kernel of $f$. 
			The other case corresponds to $x$ and $j$ having different parity. In this case the (unique) composition factor $L(j-1)$ of $M(j,\ell)$ does not belong to the same arm as the composition factor $L(x)$ in question. This implies that even if $L(j-1)$ would be contained in the image of $f$, the map would not commute with the action of $\Lambda_{n}$, and is therefore not a homomorphism.
			Therefore, in this case, we have $\dim \operatorname{Hom}_{\Lambda_n}(M(i,k),M(j,\ell))=\min(k,\ell)-\max(i,j-1)$.
			
			\item Assume that $j=i+1$.						
			First we consider $x$ such that $x\equiv i \mod 2$. Then the composition factor $L(x)$ that belongs to the top of $M(i,k)$ is contained in the same arm as the (unique) composition factor $L(i-1)$ of $M(i,k)$. Similarly, the composition factor $L(x)$ that does not belong to the socle of $M(j,\ell)$ is contained in the same arm as the composition factor $L(i+1)$ that belongs to the socle of $M(j,\ell)$. If $f$ would map $L(x)$ contained in the top of $M(i,k)$ to the (unique) composition factor $L(x)$ not contained in the socle of $M(j,\ell)$, using previous arguments leads to a similar contradiction as in the case when $j\geq i+2$. This implies that $f$ must map the composition factor $L(x)$, for $x\equiv i \mod 2$, that belongs to the top of $M(i,k)$, to the socle of $M(j,\ell)$. 
			
			Next we consider $x$ such that $x\not\equiv i \mod 2$. Then the composition factor $L(x)$ that belongs to the top of $M(i,k)$ is contained in the same arm as the composition factor $L(i)$ that belongs to the socle of $M(i,k)$.
			Similarly, if $x\neq \ell$, the composition factor $L(x)$ that does not belong to the socle of $M(j,\ell)$ is contained in the same arm as the composition factor $L(i)$ that belongs to the socle of $M(j,\ell)$. If $x=\ell$, then it follows that $L(i)\subset \operatorname{soc}(M(j,\ell))$ belongs to the lower arm of $M(j,\ell)$.
			We have five cases.
			\begin{enumerate}[$-$]
				\item  Suppose $\ell < k$ and $i\equiv \ell \mod 2$. This implies that $L(x)\subset\operatorname{upp}(M(j,\ell))$ and $x\neq \ell$. If $f$ would map $L(x)$ contained in the top of $M(i,k)$ to the (unique) composition factor $L(x)$ not contained in the socle of $M(j,\ell)$, then, in the same way as before, we see that $L(\ell+1)\subset M(i,k)$ is not contained in the kernel of $f$, which is a contradiction since $L(\ell+1)$ is not a composition factor of $M(j,\ell)$. Therefore, in this case, we have $\dim \operatorname{Hom}_{\Lambda_n}(M(i,k),M(j,\ell))=\min(k,\ell)-\max(i,j-1)$.
				
				\item Suppose $\ell < k$ and $i\not\equiv \ell \mod 2$. Then  $\operatorname{low}(M(j,\ell))$ occurs as a quotient of $M(i,k)$, giving us a homomorphism. 
				Using the same argument as above we see that this is the only possibility for $f$. Therefore, in this case, we have $\dim \operatorname{Hom}_{\Lambda_n}(M(i,k),M(j,\ell))=\min(k,\ell)-\max(i,j-1)+1$.
				
				\item Suppose $\ell >k$ and $i\equiv \ell \mod 2$. Then, $\operatorname{upp}(M(i,k))$ is isomorphic to a submodule of $M(j,\ell)$, giving us a homomorphism. Using the same argument as above we see that this is the only possibility for $f$. Therefore, in this case, we have $\dim \operatorname{Hom}_{\Lambda_n}(M(i,k),M(j,\ell))=\min(k,\ell)-\max(i,j-1)+1$.
				
				\item Suppose $\ell >k$ and $i\not\equiv \ell \mod 2$. This implies that $L(x)\subset \operatorname{low}(M(i,k))$. Then, we find that $L(k+1)\subset M(j,\ell)$ is contained in the image of $f$. This is a contradiction since $L(k+1)$ is not a composition factor of $M(i,k)$. Therefore, in this case, we have $\dim \operatorname{Hom}_{\Lambda_n}(M(i,k),M(j,\ell))=\min(k,\ell)-\max(i,j-1)$.
				\item Suppose $\ell=k$. If $L(x) \subset \operatorname{upp}(M(i,k))$, then also $L(x)\subset \operatorname{upp}(M(j,\ell))$. Using the same arguments as before leads to a similar contradiction as in the case $j\geq i+2$.  		
				If, instead, $L(x)\subset \operatorname{low}(M(i,k))$ then also $L(x)\subset \operatorname{low}(M(j,\ell))$ and the situation is similar to the previous case. Therefore, in this case, we have $\dim \operatorname{Hom}_{\Lambda_n}(M(i,k),M(j,\ell))=\min(k,\ell)-\max(i,j-1)$.
			\end{enumerate}
			\item Assume that $j\leq i$.  We have three cases.
			\begin{enumerate}[$-$]
				\item Suppose $\ell < k$. 
				Depending on the parity of $\ell-i$, either $\Omega(\ell, i-1,\ell)$ or $\Omega(\ell, i, \ell)$ is a quotient of $M(i,k)$ and this quotient is isomorphic to the lower arm of $M(j,\ell)$. This defines a homomorphism $f$ from $M(i,k)$ to $M(j,\ell)$.
				Using the same arguments as above, we see that this is the only possibility for $f$. Therefore, in this case, we have $\dim \operatorname{Hom}_{\Lambda_n}(M(i,k),M(j,\ell))=\min(k,\ell)-\max(i,j-1)+1$.
				
				\item Suppose $\ell > k$. In this case, we see that the upper arm of $M(i,k)$ is isomorphic to a submodule of $M(j,\ell)$, and $f$ can be chosen as (a scalar multiple of) the obvious embedding. Using the same argument as above, we see that this is the only possibility for $f$. Therefore, in this case, we have $\dim \operatorname{Hom}_{\Lambda_n}(M(i,k),M(j,\ell))=\min(k,\ell)-\max(i,j-1)+1$.
				
				\item Suppose $\ell=k$. In this case, we see that $M(i,k)\subset M(j,\ell)$ is a submodule, and consequently, $f$ can be chosen as (a scalar multiple of) the obvious embedding. Using the same argument as above, we see that this is the only possibility for $f$. Therefore, in this case, we have $\dim \operatorname{Hom}_{\Lambda_n}(M(i,k),M(j,\ell))=\min(k,\ell)-\max(i,j-1)+1$.\qedhere
			\end{enumerate}
		\end{enumerate} 
	\end{proof}

	\begin{lemma}\label{lemma:Ext from Delta to simple}
		If $y\leq x$, then $\operatorname{Ext}_{\Lambda_n}^m(\Delta(x),L(y))=0$, for all $m>0$. If $y>x$, then $\operatorname{Ext}_{\Lambda_n}^m(\Delta(x),L(y))\neq 0$, if and only if $m=y-x$. 
	\end{lemma}
	\begin{proof}
		Splicing the short exact sequences $$\Delta(x+1)\hookrightarrow P(x) \tto \Delta(x),$$ it follows that, at position $m$ in the projective resolution of $\Delta(x)$, we find the projective module $P(x+m)$. This means that, if $y\leq x$, then $\operatorname{Ext}_{\Lambda_n}^m(\Delta(x),L(y))=0$, for all $m>0$. At position $y-x$, we find the projective module $P(y)$, which surjects onto $L(y)$, giving rise to a non-zero extension. It is also clear that $\operatorname{Ext}_{\Lambda_n}^m(\Delta(x),L(y))=0$, for $m\neq y-x$, since there is no homomorphism from $P(m+x)$ to $L(y)$ for such $m$.
	\end{proof}
	
	With these results in hand, we can now determine between which pairs of modules in $\mathcal{D}_n$ all extensions of positive degree vanish.
	
	\begin{lemma}\label{lemma:higher ext from (i,l) to (i+1,m) vanishes}
		For all $m\geq 2$, we have $\operatorname{Ext}_{\Lambda_n}^m(M(i,k),M(i+1,\ell))=0$.
	\end{lemma}
	\begin{proof}
		By Equation~\eqref{equation: dim Ext^2}, we have
		\begin{gather*}
			\dim \operatorname{Ext}_{\Lambda_n}^2(M(i,k),M(i+1,\ell))=\\
			=\dim\operatorname{Hom}_{\Lambda_n}(\Delta_X,M(i+1,\ell)) -\dim\operatorname{Hom}_{\Lambda_n}(P_X,M(i+1,\ell))  +\dim\operatorname{Hom}_{\Lambda_n}(\Delta_Y,M(i+1,\ell)),
		\end{gather*}
		where $X, Y\subset \{1,\dots, n\}$ are subsets depending on $i$ and $k$. Lemma~\ref{lemma: dim of hom-spaces} then allows us to compute the dimensions 
		\begin{gather*}
			\dim \operatorname{Hom}_{\Lambda_n}(\Delta_X, M(i+1,\ell)),\quad \dim \operatorname{Hom}_{\Lambda_n}(P_X,M(i+1,\ell)) \quad
			\text{and}\quad \dim \operatorname{Hom}_{\Lambda_n}(\Delta_Y, M(i+1,\ell)),
		\end{gather*}
		which, of course, depend on $i, k$ and $\ell$. Summing these numbers, with signs prescribed by Equation \eqref{equation: dim Ext^2}, in each case, then yields $\dim \operatorname{Ext}_{\Lambda_n}^2(M(i,k),M(i+1,\ell))=0$. This proves the case $m=2$.
		
		For $m\geq 3$, we use Equation~\eqref{equation: dim Ext^m m>=3}, which says that $$\dim\operatorname{Ext}_{\Lambda_n}^m(M(i,k),M(i+1,\ell))=\dim\operatorname{Ext}_{\Lambda_n}^{m-2}(\Delta_X,L(i-1)\oplus L(i)),$$
		where $X=\{i,i+1,\dots,k-1\}$, if $k>i$, and $X=\{i\}$, if $k=i$. Since $x\geq i$, for all $x\in X$, Lemma~\ref{lemma:Ext from Delta to simple} guarantees that $\dim\operatorname{Ext}_{\Lambda_n}^m(M(i,k),M(i+1,\ell))=\dim\operatorname{Ext}_{\Lambda_n}^{m-2}(\Delta_X,L(i-1)\oplus L(i))=0$, for all $m\geq 3.$\qedhere
	\end{proof}

	\begin{proposition} \label{proposition: no ext from (i,k) to (j,l) if j<=i}
		If $j\leq i$, then $\operatorname{Ext}_{\Lambda_{n}}^m(M(i,k),M(j,\ell))=0$, for all $m\geq 1$.
	\end{proposition}
	\begin{proof}
		This can be proved using the same strategy as for Lemma~\ref{lemma:higher ext from (i,l) to (i+1,m) vanishes}, by using Lemma~\ref{lemma: dim of hom-spaces},~\ref{lemma: dim of hom-spaces in F(Delta)}~and~\ref{lemma:Ext from Delta to simple} as well as Equation \eqref{equation: dim Ext^1}, \eqref{equation: dim Ext^2} and \eqref{equation: dim Ext^m m>=3}.\qedhere
	\end{proof}

	\begin{proposition}\label{proposition:edge in Q_n corresponds to non-zero ext}
		Let $(i,k)$ and $(i+1,\ell)$ be such that $(i,k)\prec (i+1,\ell)$. Then, $$\operatorname{Ext}_{\Lambda_{n}}^1(M(i,k),M(i+1,\ell))\neq 0.$$ 
	\end{proposition}
	\begin{proof}
		This can be proved using the same strategy as for Lemma~\ref{lemma:higher ext from (i,l) to (i+1,m) vanishes}, by using Lemma~\ref{lemma: dim of hom-spaces}~and~\ref{lemma: dim of hom-spaces in F(Delta)} as well as Equation \eqref{equation: dim Ext^1}.
	\end{proof}

	\begin{proposition}\label{proposition:no edge in Q_n corresponds to zero ext}
		Let $(i,k)$ and $(i+1,\ell)$ be such that $(i,k)\not\prec (i+1,\ell)$. Then, $$\operatorname{Ext}_{\Lambda_{n}}^m(M(i,k),M(i+1,\ell))=\operatorname{Ext}_{\Lambda_{n}}^m(M(i+1,\ell),M(i,k))=0,$$ for all $m\geq 1$.
	\end{proposition}
	\begin{proof}
		By Lemma~\ref{lemma:higher ext from (i,l) to (i+1,m) vanishes} and  Proposition~\ref{proposition: no ext from (i,k) to (j,l) if j<=i} all extensions of degree two or higher from $M(i,k)$ to $M(i+1, \ell)$ vanish, as well as all extensions of positive degree from $M(i+1,\ell)$ to $M(i,k)$.
		That the remaining extensions vanish can be proved using the same strategy as for Lemma~\ref{lemma:higher ext from (i,l) to (i+1,m) vanishes}, by using Lemma~\ref{lemma: dim of hom-spaces}~and~\ref{lemma: dim of hom-spaces in F(Delta)} as well as Equation~\eqref{equation: dim Ext^1}.
	\end{proof}

	\begin{proposition}\label{proposition: Ext from (i,k) to (j,l) is non-zero for j>=i+2}
		If $k\neq i+1$ and $j\geq i+2$, then $\operatorname{Ext}_{\Lambda_n}^m(M(i,k),M(j,\ell))\neq 0$, for some $m\geq 1$.
	\end{proposition}
	\begin{proof}
		For  $j\geq i+3$ the result follows from Equation~\eqref{equation: dim Ext^m m>=3} together with Lemma~\ref{lemma:Ext from Delta to simple}.
		For $j=i+2$, we will use the isomorphism $$\operatorname{Ext}_{\Lambda_n}^2(M(i,k),M(j,\ell))\cong \operatorname{Ext}_{\Lambda_n}^1(K,M(j,\ell)),$$ where $K$ is the kernel of the projective cover of $M(i,k)$. Note that $K$ always contains $\Delta(i+1)$ as a direct summand, and that $\Delta(i+1)=M(i+1,i+1)$. Since $(i+1,i+1)\prec (i+2,\ell)$, for all $\ell$, Proposition~\ref{proposition:edge in Q_n corresponds to non-zero ext} implies that $\operatorname{Ext}_{\Lambda_n}^2(M(i,k),M(j,\ell))\cong \operatorname{Ext}_{\Lambda_n}^1(K,M(j,\ell))$ is non-zero.
	\end{proof}
	
	\begin{proof}[Proof of Theorem \ref{theorem:order isomorphism}]
		By applying Proposition~\ref{proposition: no ext from (i,k) to (j,l) if j<=i}, Proposition~\ref{proposition:edge in Q_n corresponds to non-zero ext}, Proposition~\ref{proposition:no edge in Q_n corresponds to zero ext} and  Proposition~\ref{proposition: Ext from (i,k) to (j,l) is non-zero for j>=i+2}, it is now clear that $(i,k) \prec (j,\ell)$ if and only if there is a non-zero extension from $M(i,k)$ to  $M(j,\ell)$.
	\end{proof}

	\subsection{Characterization of generalized tilting modules}
	Using Theorem~\ref{theorem:order isomorphism} we can now characterize the generalized tilting $\Lambda_n$-modules in $\mathcal{F}(\Delta)$ in terms of anti-chains with respect to the order $\prec$.
	
	Observe that, for any $n\geq 2$, there is an epimorphism of algebras $\Lambda_n \twoheadrightarrow \Lambda_{n-1}$, given by quotienting out the two-sided ideal generated by $e_n$, the idempotent corresponding to the vertex $n$. This epimorphism induces a functor $F:\Lambda_{n-1}\operatorname{-mod}\to \Lambda_n\operatorname{-mod}$, which is well-known to be fully faithful and exact. From this, one can deduce the following.
	
	\begin{lemma}\label{lemma:natural functor is ext-full}\cite{Dlab-Ringel,KlucznikKonig}
		For any $\Lambda_{n-1}$-modules $M$ and $N$, we have
		$$\operatorname{Ext}^k_{\Lambda_{n-1}}(M,N)\cong \operatorname{Ext}_{\Lambda_n}^k(M,N),$$
		for all $k\geq 0$.
	\end{lemma}
	
	\begin{proposition}\label{proposition:every tilting module belongs to row i and i+1 for some i}
		Assume that $T$ is a generalized tilting $\Lambda_n$-module contained in $\mathcal{F}(\Delta)$. Then there exists an index $i$, where $1\leq i\leq n-1$, such that every indecomposable direct summand of $T$ is isomorphic to either $M(i,k)$, $M(i+1,\ell)$ or $P(j)$, for some $i\leq k\leq n$, $i+1\leq \ell\leq n$ or $j<i$.
		
		In this case, we say that $T$ belongs to the $i$th tier. 
	\end{proposition}
	
	\begin{proof}
		Let $i$ be the least index such that $T$ has a non-projective summand $M(i,k)$. Consider the module $M(j,\ell)$ for some $j\geq i+2$. Then, there exists a non-zero extension from $M(i,k)$ to $M(j,\ell)$, by Proposition~\ref{proposition: Ext from (i,k) to (j,l) is non-zero for j>=i+2}. By assumption, if $j<i$, then any non-projective module $M(j,\ell)$ is not a summand of $T$. This leaves as possible summands of $T$ the modules  $M(i,k), M(i+1,\ell)$ or $P(j)$, for $j<i$. 
	\end{proof}
	
	\begin{lemma}\label{lemma:P(j) for j<i is a direct summand of T}
		Assume that $T\in\mathcal{F}(\Delta)$ is a generalized tilting $\Lambda_n$-module in the $i$th tier. Then $\bigoplus_{j=1}^{i-1}P(j)$ is a direct summand of $T$.
	\end{lemma}
	\begin{proof}
		Assume towards a contradiction that $P(j)$ is not a direct summand of $T$, for some $j<i$. Let $M$ be an indecomposable direct summand of $T$. Then $M$ is projective or isomorphic to either $M(i,k)$ or $M(i+1,\ell)$. This means that $\operatorname{Ext}^m_{\Lambda_n}(P(j),M)=0$ trivially and $\operatorname{Ext}_{\Lambda_n}(M,P(j))=0$ for any $j<i$, either because $M$ is projective or by Proposition \ref{proposition: no ext from (i,k) to (j,l) if j<=i}. This implies that the module $T\oplus P(j)$ is a generalized tilting module with $n+1$ non-isomorphic direct summands, which is a contradiction, since any such module must have exactly $n$ non-isomorphic indecomposable direct summands.
	\end{proof}
	\begin{example}
		Let $n=5$. To find generalized tilting modules in the first tier, we look for anti-chains consisting of vertices in the first and second rows. 
		
		$$
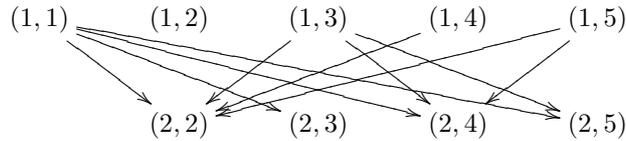
\xymatrix{
			(1,1) \ar[rd] \ar[rrd] \ar[rrrd] \ar[rrrrd] & (1,2) & (1,3) \ar[ld] \ar[rd] \ar[rrd] & (1,4) \ar[lld] & (1,5) \ar[llld] \ar[ld]\\
			& (2,2) & (2,3) & (2,4) & (2,5)	
		}$$
		We find the following anti-chains.
		\begin{gather*}
			\{(1,1), (1,2), (1,3), (1,4), (1,5)\},\quad \{(1,2), (2,2), (2,3), (2,4), (2,5)\}, \\ \{(1,2), (1, 4), (1,5),
			(2,3), (2,5)\} \quad
			\text{and}\quad \{(1,2), (1,4), (2,3), (2,4), (2,5)\}.		
		\end{gather*}
		
	\end{example}
	\begin{lemma}\label{lemma:M(i,n) or M(i+1,n) is a direct summand of T}
		Assume that $T\in\mathcal{F}(\Delta)$ is a generalized tilting $\Lambda_n$-module in the $i$th tier. Then, at least one of the modules $M(i,n)$ or $M(i+1,n)$ occurs as a direct summand of $T$.
	\end{lemma}
	\begin{proof}
		Assume towards a contraction that $T$ contains neither $M(i,n)$ nor $M(i+1,n)$ as a summand. The modules $M(i,n)$ and $M(i+1,n)$ are the only modules in the rows $i$ and $i+1$ which have a composition factor $L(n)$. Furthermore, no projective module $P(j)$ with $j<n-1$ has $L(n)$ as a composition factor. This means that $T$ restricts to a $\Lambda_{n-1}$-module, and by Lemma \ref{lemma:natural functor is ext-full}, we have $\operatorname{Ext}_{\Lambda_{n-1}}^k(T,T)\cong \operatorname{Ext}_{\Lambda_n}^k(T,T)=0$, for all $k\geq 0$. Then, $T$ is a generalized tilting $\Lambda_{n-1}$-module with $n$ summands, a contradiction.
	\end{proof}
	
	Let $1\leq i \leq n-1$ and $i< x \leq  n$. We let $Z(i,x)$ denote the following module:
	\begin{align*}
		Z(i,x):= M(i,x)\oplus M(i+1,x)\oplus \bigoplus_{j=1}^{i-1}P(j)\bigoplus_{k\in[i+1,x-1)}M(i,k)\bigoplus_{\ell\in[i+2,x-1)}M(i+1,\ell) 
		 .
	\end{align*}

	\begin{lemma}
		The module $Z(i,n)$ is a generalized tilting module, for $1\leq i\leq n-1$.
	\end{lemma}
	\begin{proof}
		It follows from Theorem~\ref{theorem:order isomorphism} that $Z(i,n)$ is self-orthogonal. Since $Z(i,n)$ contains $n$ non-isomorphic indecomposable direct summands, it follows from Corollary~\ref{corollary:theorem:RS} that it is a generalized tilting module.
	\end{proof}
	
	\begin{proposition}\label{proposition:If both M(i,n) and M(i+1,n) are direct summands then T=Z(i,n)}
		Assume that $T\in\mathcal{F}(\Delta)$ is a basic generalized tilting $\Lambda_n$-module in the $i$th tier. If both of the modules $M(i,n)$ and $M(i+1,n)$ are direct summands of $T$, then $$T=Z(i,n).$$
	\end{proposition}
	\begin{proof}
		There is a non-zero extension from the module $M(i,n)$ to the modules $M(i+1, k)$, for $k\in [i+1, n)$. This excludes $\lfloor\frac{n-i}{2}\rfloor$ modules from occurring as summands in $T$. Similarly, there is a non-zero extension from $M(i,k)$ to $M(i+1, n)$, for $k\in [i, n)$. This excludes $\lfloor\frac{n-i+1}{2}\rfloor$ modules from occurring as summands in $T$. In total, this excludes $n-i$ modules from occurring as summands in $T$. Since we have $2n-i$ indecomposable modules in total to choose from, and have excluded $n-i$ of them, the direct sum of the $n$ remaining indecomposable modules must be our module $T$. However, this direct sum is precisely $Z(i,n)$, which proves the claim. \qedhere
	\end{proof}
	
	\begin{theorem}\label{theorem:characterization of generalized tilting modules}
		Assume that $T\in\mathcal{F}(\Delta)$ is a basic generalized tilting $\Lambda_n$-module which is not equal to the characteristic tilting module. Then, there exists an integer $i$, where $1\leq i\leq n-1$, and an integer $x$, where $i<x\leq n$, such that:
		$$T=Z(i,x)\oplus \bigoplus_{k=x+1}^n M(i,k),$$
		if $x\equiv i \mod 2$, and
		$$T=Z(i,x)\oplus \bigoplus_{k=x+1}^n M(i+1,k),$$
		if $x\not\equiv i \mod 2$.
		
		In particular, there are $\frac{n(n-1)}{2}$ basic generalized tilting $\Lambda_n$-modules in $\mathcal{F}(\Delta)$ which are not equal to the characteristic tilting module. In total, there are $n(n-1)+1$ generalized tilting $\Lambda_n$-modules.
	\end{theorem}
	\begin{proof}
		It is easily verified that Theorem~\ref{theorem:order isomorphism} implies that a module of this form is self-orthogonal and that it contains $n$ non-isomorphic indecomposable direct summands. By Lemma~\ref{corollary:theorem:RS}, it follows that such a module is a generalized tilting module. Furthermore, for each pair $(i,x)$, we obtain non-isomorphic modules, which gives us $\sum_{i=1}^{n-1}n-i=\frac{n(n-1)}{2}$ basic generalized tilting $\Lambda_n$-modules in $\mathcal{F}(\Delta)$ which are not equal to the characteristic tilting module. To obtain the basic generalized tilting $\Lambda_n$-modules in $\mathcal{F}(\nabla)$, which are not equal to the characteristic tilting module, we use the simple preserving duality. 
		
		To prove that every basic generalized tilting module in $\mathcal{F}(\Delta)$, not equal to the characteristic tilting module, is of this form
		we proceed by induction on $n$. For $n=2$, we have two basic generalized tilting modules contained in $\mathcal{F}(\Delta)$, namely the characteristic tilting module and the module $P_{\Lambda_2}(1)\oplus P_{\Lambda_2}(2)=Z(1,2)$. Thus, the claim holds for $n=2$.
		
		Now, let $T$ be a basic generalized tilting module over $\Lambda_{n}$. By Proposition \ref{proposition:every tilting module belongs to row i and i+1 for some i} there exists an integer $i$, where $1\leq i\leq n-1$, such that $T$ belongs to the $i$th tier. If $T=Z(i,n)$ there is nothing to prove. Otherwise, according to Proposition~\ref{proposition:If both M(i,n) and M(i+1,n) are direct summands then T=Z(i,n)}, we may decompose $T$ as $T=T^\prime \oplus M(i,n)$ or as $T=T^\prime \oplus M(i+1,n)$, where, in both cases, $T^\prime$ is a basic generalized tilting module over $\Lambda_{n-1}$, according to Lemma \ref{lemma:natural functor is ext-full}.
		
		Assume first that $i\equiv n \mod 2$.
		
		\begin{enumerate}[$\bullet$]
			\item Suppose $T=T^\prime \oplus M(i,n)$. If $T^\prime$ is the characteristic tilting module over $\Lambda_{n-1}$, then $T$ would be the characteristic tilting module over $\Lambda_n$, contradicting our assumption. If
			$$T^\prime=Z(i,x)\oplus \bigoplus_{k=x+1}^{n-1} M(i+1,k),$$
			for some $x\not\equiv i \mod 2$,	then $T$ contains the summand $M(i+1, n-1)$, which is a contradiction as there is a non-zero extension from $M(i, n)$ to $M(i+1, n-1)$. Therefore, we must have
			$$T^\prime=Z(i,x)\oplus \bigoplus_{k=x+1}^{n-1} M(i,k),$$
			for some $1\leq i\leq n-2$ and $i<x\leq n-2$ such that $x\equiv i\mod 2$.
			\item Suppose $T=T^\prime \oplus M(i+1, n)$. If $T^\prime$ is the characteristic tilting module over $\Lambda_{n-1}$, then $i=1$ and we have a non-zero extension from $M(1,n-2)$ to $M(2, n)$.
			
			If
			$$T^\prime=Z(i,x)\oplus \bigoplus_{k=x+1}^{n-1} M(i,k),$$
			for some $x\equiv i \mod 2$, then $T$ contains the summand $M(i, n-1)$, which is a contradiction as there is a non-zero extension from $M(i, n-1)$ to $M(i+1, n)$. Therefore, we must have
			$$T^\prime=Z(i,x)\oplus \bigoplus_{k=x+1}^{n-1} M(i+1,k),$$
			for some $1\leq i\leq n-2$ and $i<x\leq n-1$ such that $x\not\equiv i\mod 2$.
		\end{enumerate}
		
		Assume instead that $i\not\equiv n \mod 2$.
		
		\begin{enumerate}[$\bullet$]
			\item Suppose $T=T^\prime \oplus M(i,n)$. If $T^\prime$ is the characteristic tilting module over $\Lambda_{n-1}$, then $T$ would be the characteristic tilting module over $\Lambda_n$, contradicting our assumption. If
			$$T^\prime=Z(i,x)\oplus \bigoplus_{k=x+1}^{n-1} M(i+1,k),$$
			for some $x\not\equiv i \mod 2$, then $T$ contains the summand $M(i+1, n-2)$, which is a contradiction as there is a non-zero extension from $M(i, n)$ to $M(i+1, n-2)$. Therefore, we must have
			$$T^\prime=Z(i,x)\oplus \bigoplus_{k=x+1}^{n-1} M(i,k),$$
			for some $1\leq i\leq n-2$ and $i<x\leq n-1$ such that $x\equiv i\mod 2$.
			\item Suppose $T=T^\prime \oplus M(i+1, n)$. If $T^\prime$ is the characteristic tilting module over $\Lambda_{n-1}$, then $i=1$ and we have a non-zero extension from $M(1,n-1)$ to $M(2,n)$.
			
			If
			$$T^\prime=Z(i,x)\oplus \bigoplus_{k=x+1}^{n-1} M(i,k),$$
			for some $x\equiv i \mod 2$, then $T$ contains the summand $M(i, n-2)$, which is a contradiction as there is a non-zero extension from $M(i, n-2)$ to $M(i+1, n)$. Therefore, we must have
			$$T^\prime=Z(i,x)\oplus \bigoplus_{k=x+1}^{n-1} M(i+1,k),$$
			for some $1\leq i\leq n-2$ and $i<x\leq n-1$ such that $x\not\equiv i\mod 2$.
		\end{enumerate} 
		This finishes the proof.
	\end{proof}
	
	We remark that $L(1)$ is a direct summand in exactly one basic generalized tilting module, namely the characteristic tilting module.
	
	\section{Exceptional sequences}
	Let $\Lambda$ be a finite-dimensional $\Bbbk$-algebra. Recall, see \cite{Bo89},
	that an indecomposable $\Lambda$-module $M$ is called  \emph{exceptional} provided that
	\begin{enumerate}[$\bullet$]
		\item $\operatorname{End}_{\Lambda}(M)\cong\Bbbk$;
		\item $\operatorname{Ext}^i_{\Lambda}(M,M)=0$, for all $i>0$.
	\end{enumerate}
	A sequence $\mathbf{M}=(M_1,\dots, M_k)$ of $\Lambda$-modules is called
	an \emph{exceptional sequence} provided that 
	\begin{enumerate}[$\bullet$]
		\item each $M_i$ is exceptional;
		\item $\operatorname{Ext}^i_{\Lambda}(M_x,M_y)=0$, for all $1\leq y< x\leq k$ and all $i\geq 0$.
	\end{enumerate}
	An exceptional sequence is called \emph{full (or complete)} if it generates the derived category. In particular, this means that it must contain at least $n$ modules, where $n$ is the number of isomorphism classes of simple $\Lambda$-modules. Indeed, suppose an exceptional sequence $\mathbf{N}$ contains $k<n$ modules. Let $\mathcal{A}\subset D^b(\Lambda)$ be the triangulated subcategory generated by $\mathbf{N}$. Passing to the Grothendieck group $K_0(\mathcal{A})$, we see that any mapping cone of a homomorphism between modules in $\mathbf{N}$ equals a linear combination of modules in $\mathbf{N}$, implying that $\operatorname{rank}K_0(\mathcal{A}) \leq k<n=\operatorname{rank}K_0(D^b(\Lambda))$. Thus, such an exceptional sequence cannot be full.
	
	\begin{proposition}\label{proposition:exceptional modoules}
		The only exceptional $\Lambda_n$-modules are the standard modules $\Delta(i)$ and the costandard modules $\nabla(i)$, for $i=1,2,\dots, n$.
	\end{proposition}
	\begin{proof}
		By Proposition \ref{proposition:characterization of self-orthogonal modules}, we know that $\Omega(i,j,k)$ is self-orthogonal exactly when  $i=j=1$ or $|i-j|=1$.
		For such $i,j$, the module $\Omega(i,j,k)$ is neither standard nor costandard if and only if $k>\max(i,j)$. In this case, there is one composition factor $L(k-1)$ in the top of $\Omega(i,j,k)$ and another composition factor $L(k-1)$ in the socle of $\Omega(i,j,k)$. This means that there is an endomorphism of $\Omega(i,j,k)$ sending the composition factor $L(k-1)$ in the top to the composition factor $L(k-1)$ in the socle. Thus $\dim \operatorname{End}_{\Lambda_n}(\Omega(i,j,k))>1$ and $\Omega(i,j,k)$ is therefore not an exceptional module.
		
		For any quasi-hereditary algebra both standard and costandard modules are self-orthogonal and have trivial endomorphism algebras, and are therefore exceptional modules. 
	\end{proof}
	\begin{theorem}\label{theorem:exceptional sequences}
		Let $\mathbf{M}$ be a full exceptional sequence of $\Lambda_n$-modules. Then $\mathbf{M}$ is of the form
		\begin{align}\tag{$\ast$}\label{equation:full exceptional sequence}
			\left(\nabla(m_1),\nabla(m_2),\dots, \nabla(m_i),L(1), \Delta(n_1),\Delta(n_2),\dots, \Delta(n_j)\right)
		\end{align}
		where
		\begin{enumerate}[$\bullet$]
			\item $i+j=n-1$;
			\item $\{m_1,m_2,\dots, m_i, n_1,n_2,\dots, n_j\}=\{2,3,\dots, n\}$;
			\item $m_1 > m_2 >\dots >m_i$ and $n_1<n_2<\dots n_j$.
		\end{enumerate}
	\end{theorem}
	\begin{proof}
		By Proposition \ref{proposition:exceptional modoules}, the standard and costandard modules are the only exceptional modules, so no other modules may be included in an exceptional sequence.
		We observe that the only non-zero homomorphisms, not equal to a scalar multiple of the identity, between standard and costandard modules are the following.
		\begin{enumerate}[$\bullet$]
			\item We may map the top of $\Delta(k)$ to the socle of $\nabla(k)$.
			\item We may map the top of $\nabla(k)$ to the socle of $\Delta(k)$.
			\item We may map the top of $\Delta(k)$ to the socle of $\Delta(k+1).$
			\item We may map the top of $\nabla(k)$ to the socle of $\nabla(k-1)$.
		\end{enumerate}
		This implies that an exceptional sequence cannot contain both $\Delta(k)$ and $\nabla(k)$ for $k\neq 1$. Further, as a full exceptional sequence must contain $n$ modules, it has to be of the form \eqref{equation:full exceptional sequence}, up to ordering. 
		
		By Proposition~\ref{proposition:non-zero extensions between F(Delta) and F(nabla)}, there is a non-zero extension from each costandard module to any standard module. This implies that any costandard module must come before any standard module in an exceptional sequence. By Theorem~\ref{theorem:order isomorphism}, or more specifically, by Proposition~\ref{proposition:edge in Q_n corresponds to non-zero ext} and Proposition~\ref{proposition: Ext from (i,k) to (j,l) is non-zero for j>=i+2}, there is a non-zero extension from $\Delta(i)$ to $\Delta(j)$, for every $i<j$. Using the simple preserving duality, we find that there is a non-zero extension from $\nabla(i)$ to $\nabla(j)$, for every $i>j$. This implies that any full exceptional sequence must be of the form \eqref{equation:full exceptional sequence}, as it must contain at least $n$ modules.
		
		Left to prove is that a sequence $\mathbf{M}=(M_1,\dots, M_n)$ of the form \eqref{equation:full exceptional sequence} actually is a full exceptional sequence. That $\operatorname{Hom}_{\Lambda_n}(M_x,M_y)=0$, for $x>y$, is clear from the first part of the proof. Furthermore, for every quasi-hereditary algebra, the following equality holds.
		$$\operatorname{Ext}^{m}(\Delta(i),\Delta(j))\cong\operatorname{Ext}^{m}(\nabla(j),\nabla(i))\cong\operatorname{Ext}^{m}(\Delta(k),\nabla(\ell))=0,$$
		for $i>j$, any $k,\ell$ and $m>0$. From this it follows that any sequence of the form \eqref{equation:full exceptional sequence} is exceptional.
		
		Left to show is that a sequence of the form \eqref{equation:full exceptional sequence} generates the derived category (as a triangulated category). We recall that any triangulated category is closed under taking kernels of epimorphisms and cokernels of monomorphisms. As the set of simple modules generate the derived category, it suffices to show that we can obtain the simple modules $L(i)$ by performing these operations on the modules in our sequence.
		
		We proceed by induction on $i=1,2,\dots, n$. The basis for the induction is clear, since $L(1)$ is included in any sequence of the form \eqref{equation:full exceptional sequence}. Next, assume that $L(i-1)$ can be obtained from our sequence. By assumption, the sequence contains either $\Delta(i)$ or $\nabla(i)$.
		In the first case, $L(i)$ is the cokernel of the inclusion $L(i-1)\hookrightarrow \Delta(i)$. In the second case, $L(i)$ is the kernel of the projection $\nabla(i)\tto L(i-1)$. This shows that any sequence of the form \eqref{equation:full exceptional sequence} is a full exceptional sequence. \qedhere
	\end{proof}
					
					%\newpage
					\vspace{5mm}
					\printbibliography
					\vspace{5mm}
					
					\noindent
					Department of Mathematics, Uppsala University, Box. 480,
					SE-75106, Uppsala, SWEDEN, 
					
					\noindent
					EPW email: {\tt elin.persson.westin\symbol{64}math.uu.se} \\
					MT email: {\tt markus.thuresson\symbol{64}math.uu.se}

				\end{document}